\documentclass[12pt,twoside]{amsart}
\usepackage{amsmath,amssymb,graphicx}
\usepackage{hyperref}

\usepackage{color}
\usepackage[svgnames]{xcolor}

\newtheorem{theorem}{Theorem}[section]
\newtheorem{lemma}[theorem]{Lemma}

\newtheorem{proposition}[theorem]{Proposition}
\newtheorem{corollary}[theorem]{Corollary}
\newtheorem{definition}[theorem]{Definition}
\newtheorem{rmrk}[theorem]{Remark}
\newtheorem{conjecture}[theorem]{Conjecture}

\DeclareMathAlphabet{\mathbfit}{OML}{cmm}{b}{it}

\makeatletter
\@addtoreset{equation}{section}
\makeatother

\newcommand{\newfig}[4] {
\bigskip
\begin{figure}[htbp]
  \centering
  \includegraphics[width=#3]{#2}
  \begin{minipage}[t]{0.80\linewidth} 
    \caption{#4}
    \protect\label{#1}
  \end{minipage}
\end{figure}
}

\newenvironment{remark}
{\begin{rmrk} \em}
{\end{rmrk}}

\newcommand{\fn} {function}

\newcommand{\R} {\mathbb{R}}
\newcommand{\C} {\mathbb{C}}

\newcommand{\Z} {\mathbb{Z}}
\newcommand{\N} {\mathbb{N}}

\newcommand{\ds} {\displaystyle}
\newcommand{\proofof}[1] {\noindent \textsc{Proof of {#1}.} }
\newcommand{\article}[3] {\textsc{{#1}}, {\itshape {#2}}, {{#3}}.}
\newcommand{\book}[3] {\textsc{{#1}}, {\itshape {#2}}, {{#3}}.}
\newcommand{\vol} {\textbf}

\newcommand{\rw} {random walk}
\newcommand{\sgn} {\mathrm{sgn}}
\newcommand{\anot} {{\alpha_o}}
\newcommand{\pup} {p_\uparrow}
\newcommand{\pdown} {p_\downarrow}
\newcommand{\pupup} {p_{\uparrow\uparrow}}
\newcommand{\loongrightarrow} {\,\xrightarrow{\hspace{27pt}}\,}

 \newcommand{\be}{\begin{equation}}
 \newcommand{\ee}{\end{equation}}

 \newcommand{\ba}{\begin{array}}
 \newcommand{\ea}{\end{array}}

 \newcommand{\bea}{\begin{eqnarray}}
 \newcommand{\eea}{\end{eqnarray}}

 \newcommand{\bl}{\begin{lemma}}
 \newcommand{\el}{\end{lemma}}

 \newcommand{\br}{\begin{remark}}
 \newcommand{\er}{\end{remark}}

 \newcommand{\bt}{\begin{theorem}}
 \newcommand{\et}{\end{theorem}}

 \newcommand{\bd}{\begin{definition}}
 \newcommand{\ed}{\end{definition}}

 \newcommand{\bcl}{\begin{claim}}
 \newcommand{\ecl}{\end{claim}}

 \newcommand{\bp}{\begin{proposition}}
 \newcommand{\ep}{\end{proposition}}

 \newcommand{\bc}{\begin{corollary}}
 \newcommand{\ec}{\end{corollary}}

 \newcommand{\bpr}{\begin{proof}}
 \newcommand{\epr}{\end{proof}}

 \newcommand{\bi}{\begin{itemize}}
 \newcommand{\ei}{\end{itemize}}

 \newcommand{\ben}{\begin{enumerate}}
 \newcommand{\een}{\end{enumerate}}

 \def \P {{\mathbb P}}
 \def \E {{\mathbb E}}

 \def \cN {\mathcal{N}}

 \def \cT {\mathcal{T}}

 \def \a {{\alpha}}
 \def \b {{\beta}}

 \def \e {{\varepsilon}}

 \def \z {{\z}}

 \def \t {{\tau}}

 \def \z {{\zeta}}

 \def \L {{\Lambda}}

\def \à {{\`{a}}}
\def \ì{{\`{\i}}}
\def \ò{{\`{o}}}
\def \è{{\`{e}}}
\def \ù{{\`{u}}}

 \def \1{\mathbbm{1}} 

 \def\var{\hbox{\rm Var}}

\begin{document}

\title[Multilayer random walk]
{Limit theorems and lack thereof for a multilayer \\ random walk mimicking 
human mobility}

\author{Alessandra Bianchi}
\address{Dipartimento di Matematica,
Universit\`a di Padova, Via Trieste 63,
35121 Padova, Italy.}\email{alessandra.bianchi@unipd.it}
\author{Marco Lenci}
\address{Dipartimento di Fisica e Astronomia, Universit\`a di Bologna,
Via Irnerio 46, 40126 Bologna, Italy \\
and Istituto Nazionale di Fisica Nucleare,
Sezione di Bologna, Viale Berti Pichat 6/2,
40127 Bologna, Italy.}
\email{marco.lenci@unibo.it}
\author{Fran\c coise P\`ene}
\address{Univ Brest, Universit\'e de Brest,
UMR CNRS 6205, LMBA, Laboratoire de Math\'ematique de Bretagne
Atlantique, 6 avenue Le Gorgeu, 29238 Brest cedex, France.}
\email{francoise.pene@univ-brest.fr}

\date{March 3, 2025}

\begin{abstract}
We introduce a continuous-time random walk model on an infinite multilayer structure inspired by transportation networks. Each layer is a copy of $\mathbb{R}^d$, indexed by a non-negative integer. A walker moves within a layer by means of an inertial displacement whose speed is a deterministic function of the layer index and whose direction and duration are random, but with a timescale that depends on the layer. After each inertial displacement, the walker may randomly shift level, up or down, independently of its past. The multilayer structure is hierarchical, in the sense that the speed is a nondecreasing function of the layer index. Our primary focus is on the diffusive properties of the system. Under a natural condition on the parameters of the model, we establish a functional central limit theorem for the $\mathbb{R}^d$-coordinate of the process. By contrast, in a class of examples where this condition is violated, we are able to determine the correct scaling of the process while proving that no limit theorem holds.
  
 
\end{abstract}

\maketitle

\section{Introduction}
\label{intro}

Multilayer networks, a.k.a.\ multiplex networks, have become ubiquitous tools in the realm of complex systems to describe situations in which elements of the same system interact in different ways. Applications are found in Statistical Physics, Computer Science, Network Theory, Sociometry, Biology, etc. (A scant list of references, also related to the other subject of this paper, random walks, includes \cite{gczm, sdga, lxl, cxa, bgb, lal}; see also the many references therein.)

One important example is that of the transportation network for human mobility: one can think of locations (cities, neighborhoods, or individual addresses) as vertices of a graph, and different ways of connecting them (walkways, urban roads, highways, railways, airways, etc.)\ as different types of edges. This results in a graph with marked edges, equivalently a \emph{multilayer graph}, where each subgraph defined by edges of the same type is regarded a \emph{layer} of the whole graph. A network of this kind admits a natural order of the layers, from the densest and typically most uniformly connected, but also slowest --- say the urban roads --- to the sparsest and fastest --- say the airways. Thus, it is reasonable to refer to the layers also as \emph{levels}, labeling them by means of nonnegative integers.

A natural way to study the connectivity properties of such a transportation network is to run a random walk on it acting differently on different levels. For example, it makes sense to prescribe the random walker to travel faster on higher levels (a flight is faster than a train ride, which is faster than a car trip, etc.)\ and with different average times, depending on the level (say, a car trip on urban roads is typically shorter than a car trip on the highway system).

The model we present here is an abstraction of the structure just discussed, whose main feature is that both the space and the number of levels are infinite. This is physically unrealistic, but mathematically convenient for studying diffusion and other large-scale properties. The reference space is not a graph, at least not a standard one, but the system can understandably be referred to as a \emph{multilayer random walk}. It is a continuous-time process on $\R^d \times \N$, where $\R^d \times \{\ell\}$ plays the role of the $\ell^\mathrm{th}$ level. To each level are associated two positive numbers: $U_\ell$, which is the speed the process maintains while on the $\ell^\mathrm{th}$ level, and $\sigma_\ell$, a time-scale for the typical displacement on the $\ell^\mathrm{th}$ level. The walk starts in some point $(x_0, L_0)$. Then an $\R^d$-valued centered random variable $\xi_0$ is drawn, with finite, positive-definite covariance matrix. The walker starts to move on the level $L_0$, with velocity $U_{L_0} \xi_0/|\xi_0|$ for a time $\sigma_{L_0} |\xi_0|$. At the end of this \emph{inertial displacement} (that is, after reaching the point $x_0 + U_{L_0} \sigma_{L_0} \xi_0$), the walker has the opportunity to shift level. If $L_0>0$ (that is, the current level is not the bottom level), the walker can go up or down one level, or stay on the same level, with probabilities, respectively, $\pup, \pdown, 1-\pup-\pdown$. If $L_0=0$, the walker will go up one level with probability $\pupup$ or else stay on level 0. These choices are independent of everything else. The new level is labeled $L_1$. 
Then the procedure repeats, in the sense that a new variable $\xi_1$ is drawn, independent of and distributed as $\xi_0$, which determines the next inertial displacement $U_{L_1} \sigma_{L_1} \xi_1$, run at speed $U_{L_1}$. After this, the level is updated again, with the rule described earlier, and so on. See Fig.~\ref{fig1}.

\newfig{fig1}{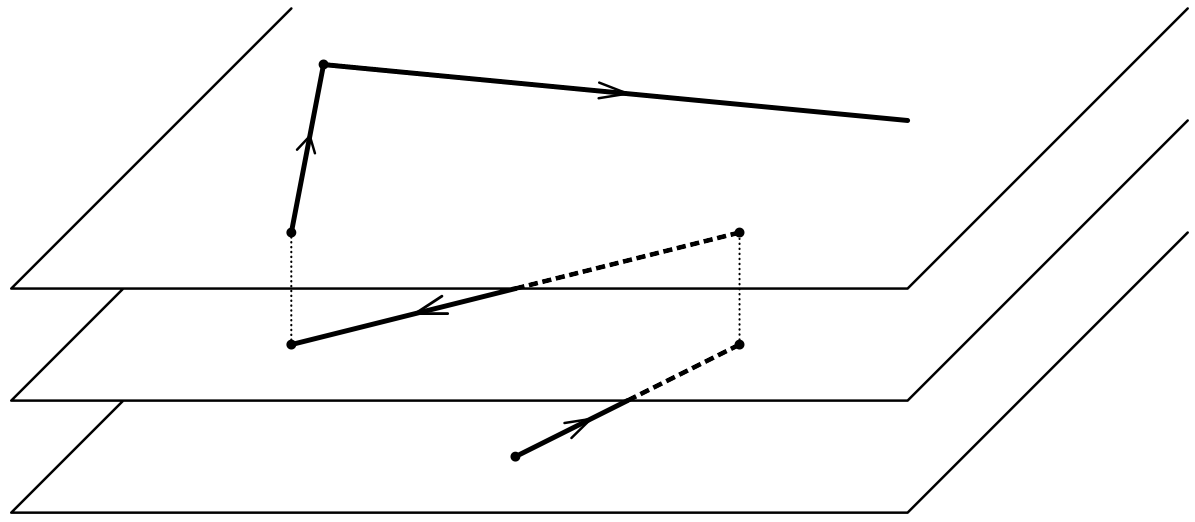}{9cm}{A trajectory of the process $W_t$, in the case $d=2$.}

In summary, our process is a continuous-time persistent \rw\ $W_t = \big(W^{(1)}_t, W^{(2)}_t \big)$ determined by the (deterministic) parameters $U_\ell, \sigma_\ell > 0$ ($\ell \in \N$) and two independent processes $(\xi_n)_{n \in \N}$, $(L_n)_{n \in \N}$. The first is a sequence of i.i.d.\ variables on $\R^d$, with zero average and finite positive-definite covariance matrices, and the second is a \rw\ on $\N$ with negative drift, because, for our model to make sense, we always assume that $\pdown > \pup > 0$. (After all, one tends to use cheaper and more widespread layers of a transport network when not too inconvenient.) In addition, according to our interpretation of the system, we assume that $U_\ell$ increases with $\ell$.

The level $W^{(2)}_t \in \N$ can well be construed as an internal state of the walker $W^{(1)}_t$ on $\R^d$. In this regard, our process has similarities with certain \emph{run-and-tumble} models for the motion of active matter that have been extensively studied in both the physical and mathematical literature; cf, e.g., \cite{dr, tss, zdk, ag, ssk, sbs, rv} and references therein --- see, in particular, the model of \cite{vvr}.

We are interested in the diffusive properties of our system, more specifically in limit theorems for the horizontal component of the process. The main results of the paper lie at two opposite ends of this spectrum:
\begin{enumerate}
  \item Under an integrability condition (cf.~Remark \ref{rk:firstcase}), $\big( W^{(1)}_t \big)_{t>0}$ satisfies the functional central limit theorem, a.k.a.\ \emph{invariance principle}, in $C([0,T])$, for all $T>0$. 
    
  \item\label{2nd-result} When the above condition is not satisfied, we give examples where any limit theorem fails. More precisely, while we identify the right scaling $n^{1/\a}$ for $W^{(1)}_t$ (in the sense that for $t>0$, $W^{(1)}_{nt}/b_n$ converges to 0 for all $b_n \gg n^{1/\a}$, while it spreads indefinitely for all $b_n \ll n^{1/\a}$) we show that $W^{(1)}_{nt}/n^{1/\a}$ does not converge. Paraphrasing the expression `strong anomalous diffusion' used in Statistical Physics \cite{cmmv, krs}, these examples exhibit quite \emph{strange} anomalous diffusion.
\end{enumerate}

The result (\ref{2nd-result}) uses, among other arguments, a computer-assisted proof (more precisely, an argument that needs a \emph{finite number} of computer-operated computations, which we do with abundant numerical precision, but not in interval arithmetic; anyone with a minimum experience with interval arithmetic software will be able to certify the computations). A byproduct of this proof is the rare occurrence of a random variable which, although naturally and quite simply defined in the context of a physical model, is not in the domain of attraction of a stable law.

\medskip

The paper is organized as follows. In Section 2 we define the model and state our main results, which are collected in Theorems \ref{th:firstcase}, \ref{THMcase2} and \ref{THMcase2-convergence}. In Section \ref{sec-sqint} we prove the functional limit theorem (Theorem \ref{th:firstcase}). In Section \ref{sec-nonsqint} we deal with a subclass of models where the assumptions of the limit theorem do not hold: we give quite a few detailed results about this case, to conclude with the proofs of Theorems \ref{THMcase2} and \ref{THMcase2-convergence}. Lastly, in Section \ref{sec-cap}, we present what amounts to be a computer-assisted proof that in many --- we believe all --- cases, only one of the alternatives of Theorem \ref{THMcase2-convergence} holds, namely, the non-convergence of the rescaled process (see also Conjecture \ref{conjecture}).

\subsection*{Acknowledgments.} 
The authors are grateful to Frank Redig for useful bibliographical suggestions.
A.B.\ was partially funded by the University of Padova through the BIRD project 239937 `Stochastic dynamics on graphs and random structures'.
M.L.\ was partially supported by the PRIN Grant 2022NTKXCX of the Ministry of University and Research (MUR), Italy. 
F.P.\ conducted this work within the framework of the Henri Lebesgue Center (ANR-11-LABX-0020-01), with the support of the Institut Brestois du Numérique et des Mathématiques (IBNM) and of the ANR project RAWABRANCH (ANR-23-CE40-0008).
This research is also part of A.B.\ and M.L.'s activities
respectively within GNAMPA and GNFM (INdAM, Italy).
The authors thank the Universities of Bologna, Brest, Padova, and the Scuola Normale Superiore di Pisa for their hospitality.

\section{Setup and main results}
\label{sec-setup}

Let $(\xi_k)_{k\in\N}$ be a sequence of i.i.d.\ random variables in $\R^d$ 
with zero average and finite, positive-definite covariance matrix $\Sigma$. For technical reasons, we also assume that $\P(\xi_k=0)=0$, although we are confident that the results presented in this paper hold as well without this assumption. Independent of this process, let $L=(L_n)_{n\in\N}$ be a Markov chain on $\N$ with the following transition probabilities:
\begin{align}
P(L_{k+1}=\ell \,|\, L_k=0) &=
\begin{cases}
  \pupup, & \mbox{if } \ell=1; \\
  1-\pupup,  & \mbox{if } \ell=0;
\end{cases} 
\\
\mbox{for } j>0, \quad
P(L_{k+1}=\ell \,|\, L_k=j) &=
\begin{cases}
  \pup, & \mbox{if } \ell=j+1; \\
  \pdown, & \mbox{if } \ell=j-1; \\
  1-\pup-\pdown, & \mbox{if } \ell=j.
\end{cases}
\end{align}
For any probability measure $\nu$ on $\N$, we denote by $P_\nu$ the law of the Markov chain $L$ with initial distribution $\nu$. A simple computation shows that $L$ possesses the unique stationary measure $\mu=(\mu_\ell)_{\ell\in\N}$ given by
\begin{equation} \label{def:stat-m}
  \mu_0 = \frac{\ds 1-\frac {\pup}{\pdown}} {\ds 1+
  \frac {\pupup-\pup}{\pdown}}, \qquad \mu_\ell = \frac{\pupup \, \pup^{\ell-1}}{\pdown^\ell}
  \, \mu_0, 
  \quad \mbox{for } \ell>0.
\end{equation}

For all $\ell \in \N$ two parameters are given, $U_\ell$ and $\sigma_\ell$, representing, respectively, the horizontal (i.e., in-layer) speed and the scale of the flight time of the process when at level $\ell$. We assume $\ell \mapsto U_\ell$ nondecreasing. Our process of interest $W := (W_t)_{t\ge 0} := \big(W^{(1)}_t , W^{(2)}_t \big)_{t\ge 0}$ is defined as follows.

First, the initial position is $(0,L_0)$, where $L_0$ has law $\nu$. (For a deterministic initial position $\ell_0$, put $\nu = \delta_{\ell_0}$.) Then, for $n\in\N$, set
\begin{equation}
  X_n := U_{L_n} \sigma_{L_n} \xi_n \,, \qquad \cT(n) := \sum_{k=0}^{n-1} 
  \sigma_{L_k} |\xi_k|
\end{equation}
(with the understanding that $\cT(0) := 0$). $X_n$ represents the $n^\mathrm{th}$ displacement of the walker, within the layer $L_n$, and $\cT(n)$ the time when it begins to take place. We refer to $n$ as the \emph{displacement time}. In order to define $W$, we need to pass from displacement time to absolute time. This is achieved by the change-of-time function
\begin{equation}\label{eq:timechange}
  \cT(s) := \int_0^s \sigma_{L_{\lfloor u \rfloor}} |\xi_{\lfloor u \rfloor}| 
  \, du,
\end{equation}
which is continuous and (strictly) increasing, except in the negligible case where $\xi_n=0$ for some $n$. For $t\ge 0$, set
\begin{equation} \label{eq:process}
  W^{(1)}_t := \int_{0}^{\cT^{-1}(t)} X_{\lfloor u \rfloor} \,du \,, 
  \qquad W^{(2)}_t := L_{\lfloor \cT^{-1}(t)\rfloor} .
\end{equation}
This completes the definition of $W$. Observe, for example, that 
\begin{equation} \label{eq:process-T}
  W_{\cT(n)} = \left(\sum_{k=0}^{n-1} X_k \,, L_n \right)
\end{equation}
as it should, at displacement time $n$. In the rest of the paper we denote by $\P_\nu := \P \otimes P_\nu$ the law of the whole process and by $\E_\nu$ its average.

We will be mostly concerned with the asymptotics of the first coordinate of our \rw. To this aim, let us introduce the process $M = (M_s)_{s\ge 0}$, where 
\begin{equation} \label{eq:genMartingale}
  M_s := \int_{0}^s X_{\lfloor u \rfloor} \, du = \sum_{k=0}^{\lfloor s\rfloor -1}X_k
 + \left(s-\lfloor s\rfloor\right)X_{\lfloor s \rfloor} \qquad (s\ge 0).
\end{equation}
Observe that 
\begin{equation} \label{eq:composition}
  W_t^{(1)} = M \circ \cT^{-1}(t) \qquad (t\ge 0).
\end{equation}
If one can prove, under suitable assumptions, that the restriction of $M$ to $\N$ is a discrete-time square-integrable martingale and, as $n\to \infty$, $\big( \cT^{-1}(nt)/n \big)_t$ converges to a multiple of the identity in a sufficiently strong sense, then one can hope to show that $\big( W_{nt}^{(1)} / \sqrt{n} \big)_t$ converges to a Brownian motion. This is the strategy of our first main result; see also Remark \ref{rk:firstcase} below.

\begin{theorem} \label{th:firstcase}
Assume
\begin{equation} \label{hyp-firstcase}
  \bar{v} := E_\mu\big[ U_{L_0}^2 \, \sigma_{L_0}^2 \big] = \sum_{\ell=0}^\infty \mu_\ell \,U_\ell^2 \, \sigma_\ell^2 < \infty.
\end{equation}
For any distribution $\nu$ on $\N$ (representing the initial distribution of the Markov chain $L$) and any $T>0$,
\begin{equation} \label{eq:th-conv1}
  \left( \sqrt{\frac{m}{\bar{v} n}} \, W^{(1)}_{nt} \right)_{t \in [0,T]}
  \overset{d}{\underset{n\to\infty}{\loongrightarrow}} 
  \left( B^\Sigma_t \right)_{t\in[0,T]}, \quad \mbox{in } C([0,T]), \quad
  \mbox{w.r.t.\ } \P_\nu\,,
\end{equation}
where $m :=\E_\mu\big[\cT(1)\big] = \E\big[ |\xi_0| \big] \sum_{\ell=0}^\infty 
  \mu_\ell \, \sigma_\ell < \infty$ and 
$\big( B^\Sigma_t \big)_{t\ge 0}$ denotes Brownian motion with covariance matrix $\Sigma$.  
\end{theorem}

\begin{remark} \label{rk:firstcase}
The assumption (\ref{hyp-firstcase}) is precisely the one one would guess to implement the above-mentioned proof strategy. In fact, since $\mu$ is the equilibrium measure of the Markov chain $L$, (\ref{hyp-firstcase}) is equivalent to saying that, for all $n$, 
\begin{equation}
\E_\mu[ |X_n|^2 ] = \mathrm{Tr}(\Sigma) \sum_{\ell=0}^\infty \mu_\ell \, 
U_\ell^2 \, \sigma_\ell^2 < \infty 
\end{equation}
(recall that $\Sigma$ denotes the covariance matrix of $\xi_n$), making $M_n = \sum_{k=0}^{n-1} X_k$ a square-integrable martingale. Moreover, by Cauchy-Schwartz and the monotonicity of $\ell \mapsto U_\ell$,
\begin{equation}
  \left( \sum_{\ell=0}^\infty \mu_\ell \, \sigma_\ell \right)^2 \le 
  \sum_{\ell=0}^\infty \mu_\ell \, \sigma_\ell^2 \le \frac 1 {U_0^2}
  \sum_{\ell=0}^\infty \mu_\ell \, U_\ell^2 \,\sigma_\ell^2 < \infty,
\end{equation}
implying $m<\infty$. This gives a functional strong law of large numbers for $(\cT(s))_s$, which, together with other arguments, leads to the sought convergence for $\big( \cT^{-1}(nt)/n \big)_t$.
\end{remark}

Assumption (\ref{hyp-firstcase}) is not only reasonable for the diffusive behavior of our \rw, but somewhat optimal, in the sense that if it fails even by a slight margin, the behavior of $W$ can be substantially different. We show this point by working in detail on the class of examples defined by the following parameters:
\begin{itemize}
  \item $d:=1$ (the levels are one-dimensional);
  \item $U_\ell:=\Lambda^\ell$, for some $\Lambda>1$ (the speed increases exponentially with the level);
  \item $\sigma_\ell \equiv 1$ (the average flight time is the same on each level);
  \item $\xi_\ell$ are standard Gaussian random variables (in particular $\var (\xi_n) = 1$);
  \item $\pup + \pdown = \pupup = 1$ (no lazy component for the walk among the levels); 
  \smallskip
  \item $1 < \frac \pdown \pup < \Lambda^2$ (if $\frac \pdown \pup > \Lambda^2$, (\ref{hyp-firstcase}) would hold).
\end{itemize}

The exponent
\begin{equation} \label{def:alpha}
  \alpha:=\frac{\log(\pdown/\pup)}{\log \Lambda} \in (0,2)
\end{equation}
will play a major role in what follows. Observe for the moment that 
\begin{equation}
  \E_\mu \big[ |X_n|^\b \big] = \E_\mu \big[ \L^{\b L_n} |\xi_n|^\b \big] = 
  \left( \mu_0 + \mu_1 \frac \pdown \pup \sum_{\ell=1}^\infty \left( 
  \L^\b \frac \pup \pdown \right)^\ell \right) \E\big[ |\xi_0|^\b \big] ,
\end{equation}
cf.~(\ref{def:stat-m}), is finite if and only if $\b<\a$.

The next two results state that there is no sequence $(a_n)_{n\in\N}$ such that
$(M_n/a_n)_n$ converges in distribution to a non-degenerate limit, and yet the right scaling for $M_n$ is $n^{1/\a}$. In the following, we will use the notation $a_n\ll b_n$, for two positive sequences, to mean that $\lim_{n\to\infty}\frac{a_n}{b_n}=0$.

\begin{theorem}[Scaling] \label{THMcase2}
Let $\nu$ be any probability on $\N$ and $(b_n)_{n\in\N}$ be a divergent sequence of positive numbers. If $b_n\gg n^{1/\a}$, then
\begin{equation}\label{upper}
\frac{M_n}{b_n}\overset{\P_\nu}{\underset{n\to\infty}{\loongrightarrow}}0\,.
\end{equation}
Furthermore, if $b_n\ll n^{\a}$,
\begin{equation}\label{lower}
\forall r>0, \qquad \limsup_{n\to\infty} \, \P_\nu\left(|M_n/b_n|>r \right) =1 \, .
\end{equation}
\end{theorem}

In Section \ref{sec-cap} we present Conjecture \ref{conjecture}, whose formulation is quite technical and irrelevant here, upon which hinges the asymptotics of $M_n/n^{1/\alpha}$. For the moment we just mention that it is equivalent to a certain variable not being $\a$-stable and that we prove it (with computer-assisted arguments) in quite a number of cases. We believe it to be always true.

\begin{theorem}[Convergence]\label{THMcase2-convergence}
The following dichotomy holds, for the limit $n \to \infty$ w.r.t.\ $\P_\nu$, for any $\nu$ (distribution of $L_0$).
\begin{enumerate}
\item \label{item1} If Conjecture \ref{conjecture} is true, then $M_n/n^{1/\a}$ does not converge in distribution. 

\item \label{item2} If Conjecture \ref{conjecture} is false, then $\big( M_{\lfloor nt\rfloor}/n^{1/\a} \big)_{t\ge 0}$ converges in the sense of finite-dimensional distributions to a symmetric stable process $(\mathcal Y_t)_{t\ge 0}$ with independent increments such that, for any $s\in\mathbb R$,  $\E[e^{i s\mathcal Y_1 }]=e^{-\widetilde c\mu_0 |s|^\a}$, with $\widetilde c>0$.
\end{enumerate}
\end{theorem}

\section{Proof of Theorem \ref{th:firstcase}}
\label{sec-sqint}

The following proof implements the ideas presented in Remark \ref{rk:firstcase}.

\begin{proof}
In view of \eqref{eq:composition}, which represents the process $W^{(1)}$ as a composition, we start by deriving a functional central limit theorem for the suitably rescaled process $M=(M_s)_{s\ge 0}$, which, by construction, is a continuous-time interpolation of the square-integrable  martingale $(M_n)_{n\in\mathbb{N}}$. Specifically, let us consider
\begin{equation}
\widetilde{M}_n(s):=\frac 1{\sqrt{n}}\int_0^{ns} X_{\lfloor u\rfloor} du
\qquad (s\ge 0).
\end{equation}
We claim that, for every $T>0$,
\begin{equation} \label{claim1}
\left( \widetilde{M}_n(s) \right)_{s\in[0,T]}\overset{d}{\underset{n\to\infty}{\loongrightarrow}} 
\left(\sqrt{\bar{v}}B_s^{\Sigma}\right)_{s\in[0,T]},\quad \mbox{in } C([0,T]) \,,
\quad \mbox{w.r.t.\ } \mathbb{P}_{\nu}\,.
\end{equation}
Setting 
\begin{equation}
v(t):=\int_0^t (U_{L_{\lfloor u\rfloor}}\sigma_{L_{\lfloor u\rfloor}})^2 \, du \,,
\end{equation}
we observe that $v(n)\Sigma$ is the conditional covariance matrix of $M_n$ given $L=(L_k)_{k\in\N}$. Moreover, conditionally to $L$, the random variables $X_i/\sqrt{v(n)}$ are independent, so we can apply \cite[Cor.~4]{Einmahl}. In detail, using the notation therein, let us implicitly define $S_{(n)}$ by the following:
\begin{equation} \label{tildeM}
\widetilde M_n(s)= \sqrt{\frac{v(n)}{n}} \, S_{(n)} \! \left( \frac{v(ns)}{v(n)} \right) .
\end{equation}
By \cite[Cor.~4]{Einmahl}, as $n\to\infty$, $(S_{(n)}(t))_t$ converges in distribution to $B^\Sigma$, in the uniform topology, provided the following condition is satisfied:
\begin{equation} \label{HH'}
\forall \eta>0, \qquad w_n:=\sum_{i=0}^{n-1} \, \E_\mu \! \left[ \left. \frac{|X_i |^2}{v(n)} \mathbf 1_{\{|X_i|>\eta \sqrt{v(n)}\}} \right| L\right] \underset{n\to \infty}{\loongrightarrow} 0\, .
\end{equation}

To this end, 
treating separately the case $U_{L_i}\sigma_{L_i}>K$
and the case $U_{L_i}\sigma_{L_i}\le K$, we observe that, for all $K>0$,
\begin{equation}
\begin{split}
w_n &= \frac 1{v(n)}\sum_{i=0}^{n-1}(U_{L_i}\sigma_{L_i})^2 \, \E_\mu \! \left[\left. |\xi_0|^2\mathbf 1_{\{U_{L_i}\sigma_{L_i}|\xi_0|>\eta \sqrt{v(n)}\}} \right|L \right] \\
&\le \frac{n}{v(n)}\frac{1}{n}\sum_{i=0}^{n-1}\left[(U_{L_i}\sigma_{L_i})^2\mathbf 1_{\{U_{L_i}\sigma_{L_i}>K\}} \, \E[\xi_0^2]+(U_{L_i}\sigma_{L_i})^2 \, \E\!\left[|\xi_0|^2\mathbf 1_{\{K|\xi_0|>\eta   \sqrt{v(n)}\}}\right]\right]\, .
\end{split}
\end{equation}
By ergodicity, and using that $\lim_{n\to \infty}\frac{v(n)}n=\bar{v}$,
it follows that, $P_\mu$-almost surely, for all $K>0$,
\begin{equation}
\begin{split}
\limsup_{n \to \infty}w_n &\le\frac{E_\mu\! \left[ (U_{L_0}\sigma_{L_0})^2\mathbf 1_{\{U_{L_0}\sigma_{L_0}>K\}} \right] \mathbb E[\xi_0^2]}{\bar{v}}\\
&\quad\quad +\frac{ E_\mu[(U_{L_0}\sigma_{L_0})^2] \, \ds \limsup_{n \to \infty} \, \E\!\left[|\xi_i|^2\mathbf 1_{\{K|\xi_i|>\eta \sqrt{v(n)}\}}\right]}{\bar{v}}\\
&\le \frac{ E_\mu \!\left[ (U_{L_0}\sigma_{L_0})^2\mathbf 1_{\{U_{L_0}\sigma_{L_0}>K\}} \right] \mathbb E[\xi_0^2]}{\bar{v}}\, .
\end{split}
\end{equation}
Taking $K\to +\infty$, we infer that \eqref{HH'} holds true $P_\mu$-almost surely and so $(S_{(n)}(t))_t$ converges in distribution to $B^\Sigma$, for the uniform topology on compact sets, with respect to $\mathbb P$. Combining this convergence with the $P_\nu$-a.s.\ convergence of $\big(\frac{v(ns)}n \big)_s$ to $ \bar{v} \,\mathrm{id}$ 
for the uniform topology on compact sets (since $\lim_{n \to \infty}\frac{v(n)}n=\bar{v}$ and using also that $\frac{v(n)}n$ is bounded), in view of~\eqref{tildeM} we conclude that, $P_\mu$-almost surely, $\widetilde M_n$ conditioned to $L$ converges in distribution to $\sqrt{\bar{v}}B^\Sigma$ (for the uniform topology, w.r.t.~$\mathbb P$).
Taking the expectation with respect to $P_\nu \ll P_\mu$ ends the proof of~\eqref{claim1}.

As a second step, let us consider the process
$\widetilde{\cT}^{-1}_n(u):= \frac{\cT^{-1}(nu)}{n}$, for $u\ge 0\,.$
We claim that 
\begin{equation} \label{claim2}
\left( \widetilde{\cT}^{-1}_n(u) \right)_{u\in[0,T]} \underset{n\to\infty}{\loongrightarrow} \left( \frac u m \right)_{u\in[0,T]}, \quad \mbox{in } C[0,T], \quad \P\mbox{-a.s.} 
\end{equation}
Indeed, by the strong law of large numbers together with the Slutsky Lemma, first note that
\begin{equation}
\forall t\ge 0, \qquad \frac{\cT(nt)}{n}=\frac 1n \int_0^{nt}\sigma_{L_{\lfloor u\rfloor}}|\xi_{\lfloor u\rfloor}| \, du \underset{n\to\infty}{\loongrightarrow} m t\,, \quad \P\mbox{-a.s.} 
\end{equation}
Taking the inverse function, we then infer (classical argument, see, e.g., \cite[Cor.~3.4.1]{WW2}) that
\begin{equation} \label{T-LLN}
\forall u\ge 0, \qquad
\widetilde{\cT}^{-1}_n(u)=\frac{\cT^{-1}(nu)}{n} \underset{n\to\infty}{\loongrightarrow} \frac u m \,,
\quad \P\mbox{-a.s.} 
\end{equation}
By a standard argument (e.g., \cite[Cor.~3.2.1]{WW2}), the above pointwise convergence implies the uniform convergence in $C[0,T]$.

Finally, taking together the convergences \eqref{claim1} and \eqref{claim2}, and thanks to \eqref{eq:composition}, the main assertion \eqref{eq:th-conv1} follows from the application of \cite[Thm.~3.9]{Bill} combined with \cite[Lemma at p.~151]{Bill}.
\end{proof}

\section{Proofs of Theorems \ref{THMcase2} and \ref{THMcase2-convergence}} 
\label{sec-nonsqint}

When we deal with a stochastic process indexed by time that, when rescaled with the square root of time, does not converge to a Brownian motion together with its (most important) moments, we enter the realm of \emph{anomalous diffusion} \cite{krs, zdk}. In this business, it is known that a wide array of behaviors may occur, some of which are rather peculiar and even puzzling, such as that of the so-called \emph{L\'evy-Lorentz gas} \cite{bfk}, especially if compared to the \emph{L\'evy walk} with the same flight distribution \cite{bcv, bcll, acor, blp, sal}.

The proofs of Theorems \ref{THMcase2} and \ref{THMcase2-convergence}, and the results leading to them, will reveal an even stranger behavior for our model, when the condition for normal fluctuation is not verified.

\subsection{Strategy of the proofs}
From now on we will assume that the initial measure of the process $L=(L_k)_{k\in\mathbb{N}}$ is $\delta_{0}$ and denote the corresponding probability measure by $P_0 := P_{\delta_0}$ (and $\P_0$ when the whole process is considered). A general argument from ergodic theory, Lemma~\ref{lemZwei} below, will show that our results will hold as well w.r.t.\ any initial measure $\nu$ on $\N$.

For any $n\in\N$, denote
\begin{equation} \label{def:Vn}
V_n:=\sum_{j=0}^{n-1}\Lambda^{2L_j} \,,
\end{equation}
with the convention that $V_0:=0$, and observe that $V_n$ is precisely the conditional variance of $M_n$ given $L$.
The proofs of Theorems \ref{THMcase2} and \ref{THMcase2-convergence} will then use, as a fundamental tool, the convergence (or nonconvergence) of $(V_n)_{n\in\mathbb{N}}$ under suitable rescaling.

For example, the assertion \eqref{upper} of Theorem \ref{THMcase2} will follow from the convergence in probability of $(V_n/b_n^2)_{n\in\mathbb{N}}$ to 0, which we will show to occur for any $b_n\gg n^{1/\alpha}$, while \eqref{lower} can be derived similarly using that $V_n/n^{1/\alpha}$ is nonzero with positive probability, uniformly in $n$.
Moreover, in the particular case where the $\xi_k$ are Gaussian, the characteristic function of $M_n/n^{1/\a}$ in $\theta$ coincides with the Laplace transform of $V_n/n^{2/\a}$ in $\theta^2/2$, cf.~\eqref{conditionalcarfunc}, providing a further tool for analyzing the dichotomy asserted in Theorem \ref{THMcase2-convergence}.

Sections \ref{SecStableZ}-\ref{subs-pf-prop-key} below are devoted to the study of the asymptotic behavior of $(V_n)_{n\in\mathbb{N}}$, with special attention to the distribution of the process at its first return time to 0, which we will call $Z:=V_{\tau_0}$. In particular, in Section~\ref{SecStableZ} we will focus on the peculiar ``random stability'' property of $Z$ and in Section~\ref{SecCarfuncZ} we will characterize the order of magnitude at 0 of the characteristic function of $Z$, exploring its consequences for the behavior of $(V_n)_{n\in\mathbb{N}}$. The key proposition in this regard, Proposition \ref{asymptphi}, will be proved in Section~\ref{subs-pf-prop-key}. Theorems \ref{THMcase2} and \ref{THMcase2-convergence} will be  proved in Section~\ref{sec:proofCase2}, using all the intermediate results established earlier.

We conclude this section with the following lemma, which follows from a useful argument by Zweim\"uller \cite{Zwei}. 
\begin{lemma} \label{lemZwei}
Let $\nu$ be any initial probability measure (for $L$) on $\N$ and $(b_n)_{n\in\N}$ a diverging sequence of positive real numbers. Then, the convergence in distribution of $(M_n/b_n)_{n\in\N}$ to some random variable w.r.t.\ $\P_0$ is equivalent to the convergence in distribution to the same random variable w.r.t.\ $\P_\nu$.
Similarly, the convergence in distribution of $(V_n/b_n)_{n\in\N}$ to some random variable w.r.t.\ $P_0$ is equivalent to the convergence in distribution to the same random variable w.r.t.\ $P_\nu$. 
\end{lemma}

\begin{proof}
Observe that the dynamical system associated with the canonical process $(L_n,\xi_n)_{n\in\mathbb N}$,
endowed with the shift $F$ and the invariant measure $\mathbb P_\mu$,
is ergodic. Moreover, as $n\to\infty$,
\begin{equation} 
\frac{\left|M_n\circ F- M_n\right|}{b_n} \le \frac{|X_0|+|X_n| }{b_n} \,,
\end{equation}
converges in distribution to 0 with respect to $\mathbb P_\mu$ (by $F$-invariance of $\mathbb P_\mu$). 
The claimed convergence then follows from~\cite[Thm.~1]{Zwei}, applied with $P=\P_\nu$, $m=\P_\mu$ and $R_n=M_n$, with values in the Banach space $\mathbb R$.
The same reasoning clearly applies to the process $(V_n/b_n)_{n\in\N}$
considering the ergodic dynamical system associated with the canonical process $(L_n)_{n\in\mathbb N}$, endowed with the shift $F$ and the invariant measure $P_\mu$. The claimed convergence then follows from~\cite[Thm.~1]{Zwei}, applied with $P=P_\nu$, $m=P_\mu$ and $R_n=V_n$.
\end{proof}

\subsection{The conditional variance \texorpdfstring{$V_n$}{Vn}}\label{SecStableZ}
For any $n\in\mathbb{N}$, let $N_n(0)$ be the local time of 
$L$ at $0$, that is
\begin{equation}
N_n(0):=\#\{j=1,\ldots, n : L_j=0\}\, ,
\end{equation}
and let $\tau_0^{(n)}$ denote the $n^\mathrm{th}$ return time of the walk  to $0$, 
defined inductively by
\begin{equation}
\tau_0^{(0)}:=0\,, \qquad  \tau_0^{(n)} := \min \{j>\tau_0^{(n-1)} : L_j=0\} \quad
(n>0) \,.
\end{equation}
Observe that $\big( \tau_0^{(n)} \big)_{n\in\mathbb{N}}$ has i.i.d.\ increments, and
\begin{equation}
\tau_0^{(N_{n}(0))}\le n< \tau_0^{(N_{n}(0)+1)}\,.
\end{equation}
As a consequence, since $n \mapsto V_n$ is increasing by definition, cf.~\eqref{def:Vn}, we have
\begin{equation}\label{Vtau}
V_{\tau_0^{(N_{n}(0))}}\le V_n<V_{\tau_0^{(N_{n}(0)+1)}} \,.
\end{equation}
Furthermore,  $\big( V_{\tau_0^{(n)}} \big)_{n\in\mathbb{N}}$ also has i.i.d.\ increments $(Z_k)_{k\in\Z^+}$, each of them distributed as
\begin{equation} \label{def:Z}
Z:=V_{\tau_0}=\sum_{j=0}^{\tau_0-1}\Lambda^{2L_j} \,,
\end{equation}
where $\tau_0$ is short for $\tau_0^{(1)}$, an abbreviation that we will use repeatedly hereafter. In other words, for all $n>0$, 
\begin{equation}\label{Vtaum}
V_{\tau_0^{(n)}}=\sum_{k=1}^n Z_k \,.
\end{equation}

In view of the previous equations, we will first focus on
the asymptotic behavior of $V_{\tau_0^{(n)}}$, as $n \to \infty$, 
and then use the $P_0$-a.s.\ convergence of $N_n(0)/n$ to $\mu_0$, 
and correspondingly the $P_0$-a.s.\ convergence of $\tau_0^{(n)}/n$ 
to $1/\mu_0$, to infer the asymptotic behavior of $V_n$.
We start with some additional notation that will allow us to write
$V_{\tau_0^{(n)}}$ in a convenient way.

Let $\cN$ be the local time at $1$ of $L$ during its first excursion out of 0 and back, that is
\begin{equation}
\cN:= \#\{j=1,\ldots,\tau_0 : L_j=1\} .
\end{equation}
Notice that, during the time interval $\{1,\ldots,,\tau_0\}$, there are exactly $\cN-1$ excursions from 1 to 1 going up, that is,
\begin{equation}
\#\{j=1,\ldots,\tau_0 : L_j=1, L_{j+1}=2\} = \cN-1 \,,
\end{equation}
see Fig.~\ref{fig2}.
Since $\pupup=1$ by assumption, one sees that $\cN\ge 1$ $P_0$-a.s.,
and that $\cN$ has geometric distribution of parameter $\pdown$.

Let us consider the sequence $\big(\tau_1^{(k)} \big)_{k\in\mathbb{N}}$ of return times to 1 of $L$, defined recursively by
\begin{equation} \label{def-tau1k}
\tau_1^{(0)}:=1, \qquad
\ \tau_1^{(k)} := \min \{j>\tau_1^{(k-1)} : L_j=1\}\, \quad (k> 0).
\end{equation}
From \eqref{def:Z}, and using the above notation together with the facts that $L_0=0$ and $L_{\tau_1^{(\cN)}}=1$, we have
\begin{equation} \label{EqZ}
\begin{split}
Z &= \sum_{j=0}^{\tau_0-1}\Lambda^{2L_j}
=1+ \sum_{k=1}^{\cN-1}\Lambda^2 \sum_{j=\tau_1^{(k-1)}}^{\tau_1^{(k)}-1}\Lambda^{2(L_j-1)} + \Lambda^2 \\
&= 1+\Lambda^2 +\Lambda^2\sum_{k=1}^{\cN-1} Z_k\, ,
\end{split}
\end{equation}
where $Z_k:=\sum_{j=\tau_1^{(k-1)}}^{\tau_1^{(k)}-1} \Lambda^{2(L_j-1)}$ ($k=1, \ldots, \cN-1$) are i.i.d.\ random variables with the same distribution as $Z$. Fig.~\ref{fig2} helps illustrate this point.

\newfig{fig2}{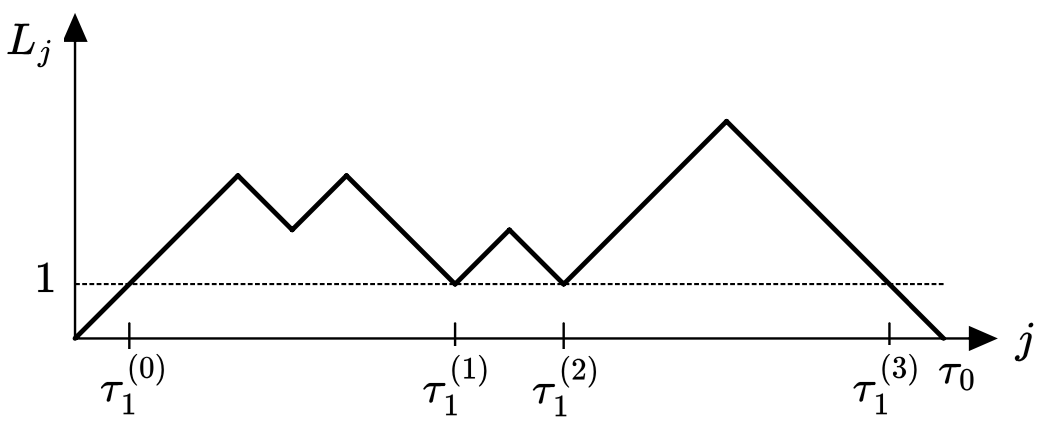}{9cm}{A realization of $L_j$, for $j=0, \ldots, \tau_0$. The figure illustrates the definitions \eqref{def-tau1k} and Eqn.~\eqref{EqZ} (with $\cN=4$).}

The key identity \eqref{EqZ} can be seen as a ``random stability'' property for the distribution of $Z$, the randomness being related to the random number $\cN-1$ of copies of $Z$ appearing in the formula. The identity \eqref{EqZ} can be transformed into an equivalent identity for $\varphi_Z$, the characteristic function of $Z$. We derive it recalling the assumption $\pdown+\pup=1$. For $\theta\in\R$,
\begin{equation} \label{EQphi}
\begin{split}
\varphi_Z(\theta) &=e^{i(1+\Lambda^2 )\theta}\, \E_0 \!\left[ \varphi_Z(\Lambda^2\theta)^{\cN-1}\right] \\
&= e^{i(1+\Lambda^2) \theta}\sum_{m=1}^\infty \varphi_Z(\Lambda^2\theta)^{m-1} \, \pup^{m-1} \pdown \\
&= e^{i(1+\Lambda^2)\theta} \, \frac{\pdown}{1-\pup\varphi_Z(\Lambda^2\theta)}\, .
\end{split}
\end{equation}

\begin{remark}
It follows from Eqn.~\eqref{EQphi} that $\varphi_Z$ is the unique fixed point of the map $\mathcal F$ given by
\begin{equation}
\mathcal F(\psi)(\theta):= e^{i(1+\Lambda^2) \theta}\frac{\pdown}{1-\pup\psi(\Lambda^2\theta)} 
\quad (\theta\in\R) \,,
\end{equation}
which defines a contraction on the set $\mathcal S$ of characteristic functions endowed with the sup norm. 
In fact, for every $\psi_1,\psi_2 \in \mathcal S$,
\begin{equation}
\left\| \mathcal F(\psi_1)-\mathcal F(\psi_2) \right\|_\infty =\pdown \pup \left\| 
\frac{\psi_1-\psi_2}{(1-\pup\psi_1)(1-\pup\psi_2)} \right\|_\infty \le \, \frac \pup \pdown \, \|\psi_1-\psi_2\|_\infty\, ,
\end{equation}
where the inequality follows from the facts that $|1-\pup\psi_1|\, |1-\pup\psi_2|\ge (1-\pup)^2=\pdown^2$ and $\pup<\pdown$ (assumed). In particular, for any characteristic function $\psi$, the above implies 
the uniform convergence of $\mathcal F^n(\psi)$ to $\varphi_Z$, with exponential rate.
\end{remark}

The proof of Theorem~\ref{THMcase2} will strongly rely on Eqn.~\eqref{EQphi} through the intermediate results stated in the following sections.

\subsection{The characteristic function of \texorpdfstring{$Z$}{Z}}\label{SecCarfuncZ}
Let us rewrite \eqref{EQphi} as
\begin{equation} \label{1-phi}
1-\varphi_Z(\theta)=\frac{\pdown(1-e^{i(1+\Lambda^2) \theta})
+\pup(1-\varphi_Z(\Lambda^2\theta))}{\pdown+\pup(1-\varphi_Z(\Lambda^2\theta))}\, .
\end{equation}
Recall the definition \eqref{def:alpha} of $\a$ and set
\begin{equation} \label{def:ell}
\ell(\theta):=\frac{1-\varphi_Z(\theta)}{|\theta|^{\a/2}} \quad (\theta \in \R \setminus \{0\}) \,.
\end{equation}
Since $Z$ is real-valued, clearly $\varphi_Z(-\theta) = \overline{\varphi_Z(\theta)}$ and $\ell(-\theta) = \overline{\ell(\theta)}$.
Since $\Lambda^\alpha=\frac{p_{\downarrow}}{p_{\uparrow}}$, the identity \eqref{1-phi} can be expressed as the following equation for $\ell(\theta)$:
\begin{equation} \label{fixed-ell0}
\ell(\theta) = \frac{\ell(\Lambda^2\theta) + \ds \frac{1 - e^{i(1+\Lambda^2) \theta}}{|\theta|^{\a/2}} } 
{|\theta|^{\a/2} \, \ell(\Lambda^2\theta) + 1} \,.
\end{equation}
The next result characterizes the behavior of the characteristic function of $Z$ around $0$.

\begin{proposition}\label{asymptphi}
There exists a bounded, non-identically null function $\R \setminus \{0\} \ni \theta\mapsto c_\theta$
such that 
\begin{equation}\label{def:ctheta}
c_\theta=\lim_{n \to \infty}\ell(\theta/\Lambda^{2n})\, .
\end{equation}
More precisely, there exists $\theta_0>0$ such that, for all $\theta \ne 0$, $-\theta_0 \le \theta \le \theta_0$, 
\begin{equation}
\lim_{n \to \infty}(1-\varphi_Z(\theta/\Lambda^{2n}))\Lambda^{n\alpha}= c_\theta |\theta|^{\alpha/2}\,,
\end{equation}
where the above convergence is exponentially fast in $n$, uniformly in $[-\theta_0,0) \cup (0,\theta_0]$.
\end{proposition}

Before proving Proposition~\ref{asymptphi}, let us state two of its consequences, which ensure that $n^{2/\a}$ is the correct scaling for $V_n$.

\begin{corollary} \label{cor:ell-bound}
The function $\ell$ given in \eqref{def:ell} is uniformly bounded.
\end{corollary}

\begin{proof}
By the symmetry properties of $\ell$, we can assume w.l.g.\ that $\theta \in \R^+$. Let us initially restrict to 
$\theta\in (0,\theta_0]$. There exists $n\ge 0$ such that, for $s := \theta\Lambda^{2n} \in (\theta_0/\Lambda^2,\theta_0]$. By definition \eqref{def:ell} and Proposition~\ref{asymptphi}, there exist $\rho\in(0,1)$ and $C_0, C_1>0$ such that
\begin{equation} 
\begin{split}
|\ell(\theta)| &= s^{-\a/2} \Lambda^{n\alpha} \left|1-\varphi_Z(s/\Lambda^{2n}) \right| \\
&\le s^{-\a/2} \left(c_s s^{\a/2}+C_0\rho^n \right) \\
&\le (\theta_0/\Lambda^2)^{-\a/2} \left(C_1 \theta_0^{\a/2}+C_0\right)\, .
\end{split}
\end{equation}
On the other hand, for $\theta>\theta_0$, $|\ell(\theta)|\le\frac{1+|\varphi_Z(\theta)|}{\theta^{\a/2}}\le 2\theta_0^{-\a/2}$, concluding the proof.
\end{proof}

\begin{corollary}\label{biggamma}
For any sequence $(a_n)_{n\in\N}$ such that $\ds \lim_{n \to \infty}a_n=0$, 
\begin{equation} \label{goal4.5}
\frac{a^2_n} {n^{2/\a}} \sum_{k=1}^n Z_k \overset{P_0}{\underset{n\to\infty}{\loongrightarrow}}0 \,, \qquad
\frac{a^2_n} {n^{2/\a}} V_n \overset{P_0}{\underset{n\to\infty}{\loongrightarrow}}0\,.
\end{equation}
\end{corollary}
\begin{proof}[Proof of Corollary~\ref{biggamma}]
It follows from Corollary \ref{cor:ell-bound} that for any $\theta\in\mathbb{R}$,
\begin{equation}
\begin{split}
	\left|E_0\!\left[e^{i\theta n^{-2/\a} a_n^2 \sum_{k=1}^n Z_k}\right]\right|
    &\le \left|\varphi_Z(\theta a_n^2/n^{2/\a})\right|^n\le \left(1+\Vert \ell\Vert_\infty \, |\theta a_n^2/n^{2/\a }|^{\a/2}\right)^n \\
	&\le e^{n\Vert \ell\Vert_\infty \, |\theta a_n^2/n^{2/\a }|^{\a/2}}=e^{\Vert\ell\Vert_\infty \, |\theta|^{\alpha/2} a_n^{\alpha }}\, ,
\end{split}  
\end{equation}
which tends to $1$ as $n \to \infty$, proving the left part of \eqref{goal4.5}.

As for the limit on the right of \eqref{goal4.5}, using the fact that $n\mapsto V_n$ is increasing and the properties described in \eqref{Vtau}-\eqref{Vtaum}, we observe that, for every $\eta>0$ and $\rho>1$,
\begin{equation} \label{epsVn}
\begin{split}
P_0(a^2_nV_n/n^{2/\a}>\eta) &\le 
P_0\!\left(\tau_0^{(\lfloor \rho n/E_0[\tau_0]\rfloor)}<n\right)+
P_0\!\left(a^2_n \, V_{\tau_0^{(\lfloor \rho n/E_0[\tau_0]\rfloor)}}/n^{2/\a}>\eta\right) \\
&\le P_0\!\left(\tau_0^{(\lfloor \rho n/E_0[\tau_0]\rfloor)}<n \right)+
P_0\!\left(a^2_n \, \lfloor \rho n/E_0[\tau_0]\rfloor^{-2/\a} \!\! \sum_{k=1}^{\lfloor\rho n/E_0[\tau_0]\rfloor} \!\! Z_k>\eta'\right) ,
\end{split}  
\end{equation}
with $\eta'= \left(E_0[\tau_0]\right)^{2/\alpha}\eta/\rho$. 
Since $\tau_0^{(m)}$ is a sum of $m$ i.i.d.\ random variables with the same distribution as $\tau_0$,  by the strong law of large numbers, $\tau_0^{(m)}/m$ converges $P_0$-a.s.\ to $E_0[\tau_0]$, as $m\to\infty$, and the first probability on the rightmost term of \eqref{epsVn} vanishes. Applying the previous convergence result, and considering the arbitrariness of $a_n\to 0$ therein, the second probability on the rightmost term of \eqref{epsVn} also vanishes, thereby concluding the proof of the second assertion of \eqref{goal4.5}.
\end{proof}

We now deduce the convergence in distribution for a lacunary subsequence of $(n^{-2/\a} \sum_{k=1}^n Z_k)_n$.

\begin{corollary} \label{CoroCase2}
Let $\widetilde{Z}$ be a random variable with characteristic function $\varphi_{\widetilde{Z}}(\theta)= e^{-c_\theta|\theta|^{\a/2}}$. Then
\begin{equation}\label{conv:Z}
\Lambda^{-2n}\sum_{k=1}^{\lfloor\Lambda^{\alpha n}\rfloor} Z_k
\overset{d}{\underset{n\to\infty}{\loongrightarrow}} \widetilde{Z} \,,
\quad \mbox{w.r.t. } P_0\,.
\end{equation}
Furthermore, there exists $\eta>0$ such that
\begin{equation}\label{positive-var}
\liminf_{n \to \infty} P_0 \!\left(V_n/n^{2/\a}>\eta\right)>0\, .
\end{equation}
\end{corollary}

\begin{proof}
As previously done, w.l.g.\ consider $\theta > 0$. Since the $Z_k$ are i.i.d.\ with characteristic function $\varphi_Z$, 
\begin{equation}
E_0 \!\left[e^{i\theta\Lambda^{-2n}\sum_{k=1}^{\lfloor\Lambda^{\alpha n}\rfloor} Z_k}\right]= \left( \varphi_Z(\theta/\Lambda^{2n}) \right)^{\lfloor\Lambda^{\alpha n}\rfloor}\, .
\end{equation}
In the notation introduced in Proposition \ref{asymptphi}, let us fix $n_0$ such that $s=\theta/\Lambda^{2n_0} \in (0,\theta_0]$. By the definition \eqref{def:ctheta}, $c_s=c_\theta$. Proposition \ref{asymptphi} implies that
\begin{equation}
\begin{split}
E_0 \! \left[e^{i\theta\Lambda^{-2n} \sum_{k=1}^{\lfloor\Lambda^{\alpha n}\rfloor} Z_k}\right]
&= \left( \varphi_Z(s/\Lambda^{2(n-n_0)}) \right)^{\lfloor\Lambda^{\alpha n}\rfloor} \\
&=\left(e^{-\Lambda^{-\alpha (n-n_0)}(|s|^{\alpha/2}c_{s}+a_n)}\right)^{\lfloor\Lambda^{\alpha n}\rfloor} \\
&= e^{-c_{\theta}|\theta|^{\alpha/2} -\Lambda^{\alpha n_0}a_n} \, ,
\end{split}
\end{equation}
for some $a_n \in \C$ such that $\ds \lim_{n \to \infty}a_n=0$. Taking this limit we obtain
\begin{equation}
\lim_{n\to\infty} E_0 \!\left[e^{i\theta\Lambda^{-2n}\sum_{k=1}^{\lfloor\Lambda^{\alpha n}\rfloor} Z_k}\right]=e^{-c_\theta \theta^{\a/2}}\, .
\end{equation}
Since $c_\theta$ is bounded, the function $e^{-c_\theta \theta^{\a/2}}$ is continuous at 0 and \eqref{conv:Z} follows.

To prove the second statement of the corollary, we observe that \eqref{conv:Z} corresponds to the convergence in distribution of the sequence $\big( \Lambda^{-2n} V_{\tau_0^{(\Lambda^{\alpha n})}} \big)_{n\in\N}$, cf.~\eqref{Vtaum}, to a non-constant random variable. We proceed as in \eqref{epsVn}: since $n\mapsto V_n$ is increasing and by \eqref{Vtau}-\eqref{Vtaum} we have that for all $\eta>0$ and $\rho<1$,
\begin{equation} \label{lastdisp1}
\begin{split}
    P_0(V_n/n^{2/\a}>\eta) &\ge P_0\!\left(n^{-2/\a}V_{\tau_0^{(\lfloor\rho n/E_0[\tau_0]\rfloor)}}>\eta\right) - P_0\!\left(\tau_0^{(\lfloor \rho n/E_0[\tau_0]\rfloor)}>n \right) \\
	&\ge P_0\! \left(\left(\frac{\Lambda^{\alpha(m_n+1)}E_0[\tau_0]}\rho\right)^{-2/\a} \!\! V_{\tau_0^{(\lfloor\Lambda^{\alpha m_n}\rfloor)}}>\eta\right) - P_0\!\left(\tau_0^{(\lfloor \rho n/E_0[\tau_0]\rfloor)}>n
	\right)\, ,
\end{split}
\end{equation}
where $m_n:=\left\lfloor \frac{\log (\rho n / E_0[\tau_0]) }{\alpha\log \Lambda } \right\rfloor$, 
so that $\Lambda^{\alpha m_n}\le \rho n / E_0[\tau_0] < \Lambda^{\alpha (m_n+1)}$.
By the strong law of large numbers on $\big( \tau_0^{(m)}/m \big)_m$, we conclude that the second term in the last line of \eqref{lastdisp1} vanishes $P_0$-almost surely. Thus, for all $\e>0$ and all correspondingly large $n$, we get
\begin{equation}
P_0(V_n/n^{2/\a}>\eta) \ge P_0 \!\left(\Lambda^{-2m_n}V_{\tau_0^{(\Lambda^{\alpha m_n})}}>\Lambda^2 \eta\left(E_0[\tau_0]/\rho\right)^{2/\a}\right)-\e\, ,
\end{equation}
and from \eqref{conv:Z} it follows that
\begin{equation}
\liminf_{n \to \infty} P_0 \!\left(V_n/n^{2/\a}>\eta\right)
\ge P_0 \! \left(\widetilde{Z} >\Lambda^2 \eta\left(E_0[\tau_0]/\rho\right)^{2/\a}\right)\,.
\end{equation}
Since $\widetilde{Z}$ is non-degenerate, there exists $\eta>0$ such that the above r.h.s.\ is positive, yielding \eqref{positive-var}.
\end{proof}

\begin{remark}\label{rem-measure}
We emphasize that the statements of Corollary \ref{CoroCase2} are still valid if the law $P_0$ of $L$ is replaced with $P_\nu$ (once again, $\nu$ denotes a generic distribution of $L_0$). More precisely, the fact that \eqref{conv:Z} implies the convergence in distribution w.r.t.\ $P_\nu$ of the sequence $\big( \Lambda^{-2n} V_{\tau_0^{(\lfloor \Lambda^{\alpha n} \rfloor)}} \big)_{n\in\N}$ to $\widetilde Z$ follows from an application of Lemma \ref{lemZwei}, while \eqref{positive-var} can be deduced by replacing $P_0$ with $P_\nu$ in its proof, line by line.
\end{remark}

In this model, the fact that $\R^+\ni\theta\mapsto c_\theta$ is constant or not 
plays a major role in view of Theorem~\ref{THMcase2-convergence}.
The next corollary states that if $c_\theta$ is not a constant for all $\theta>0$, then the distribution of $Z$ does not belong to the normal domain of attraction 
of a stable distribution.
\begin{corollary} \label{CoroCase3}
If $\R^+\ni\theta\mapsto c_\theta$ is not constant, then $n^{-2/\alpha} \sum_{k=1}^n Z_k$ does not converge in distribution, as $n \to \infty$. 
\end{corollary}
\begin{proof}
Assume that $n^{-2/\alpha} \sum_{k=1}^n Z_k$ converges. Since the $Z_k$ are i.i.d., the limit must be an $(\alpha/2)$-stable random variable. On the other hand, it must also be the same variable $\widetilde{Z}$ introduced in Corollary \ref{CoroCase2}, which cannot be stable unless $c_\theta$ is constant in $\theta \in \R^+$ \cite[Thm.~2.2.2]{IL}.
\end{proof}

\subsection{Proof of Proposition \ref{asymptphi}} \label{subs-pf-prop-key}
Let $\varepsilon \in(0,1)$ be such that $r:=(1-\varepsilon)^{-1}\Lambda^{-\alpha }<\Lambda^{-\a/2}<1$, and $\theta_0>0$ such that $\sup_{(0,\theta_0]} |1-\varphi_Z|\le\varepsilon \Lambda^\alpha$. Once again, it suffices to restrict our attention to $\theta>0$. Consider first $\theta \in (0,\theta_0]$. Denoting $a_n := 1-\varphi_Z(\theta/\Lambda^{2n})$, we can rewrite \eqref{1-phi} as
\begin{equation}
a_{n+1}=\frac{1-e^{\frac{i(1+\Lambda^2) \theta}{\Lambda^{2n+1}}} + \Lambda^{-\alpha}a_n}
{1+\Lambda^{-\alpha }a_n}\, ,
\end{equation}
having noted that $\pup/\pdown=\Lambda^{-\alpha} \in( 1/\Lambda^2 ,1)$. It follows that
\begin{equation}
|a_{n+1}|\le (1-\varepsilon)^{-1}\left(\frac{(1+\Lambda^2) \theta}
{\Lambda^{2(n+1)}}+\Lambda^{-\alpha }|a_n|\right)
\end{equation}
and, iterating,
\begin{equation}
\begin{split}
|a_n| &\le (1-\varepsilon)^{-n}\Lambda^{-n\alpha }|a_0|+\sum_{m=0}^{n-1}
(1-\varepsilon)^{-1-m}\Lambda^{-\alpha m}
\frac{(1+\Lambda^2) \theta}{\Lambda^{2(n-m)}} \\
&\le r^n|a_0|+(1+\Lambda^2) \Lambda^{-2n} \theta\, (1-\varepsilon)^{-1}
\frac{(1-\varepsilon)^{-n} \Lambda^{2n(1-\a/2)}-1}{(1-\varepsilon)^{-1}
\Lambda^{2(1-\a/2)}-1)} \\
&\le r^n |a_0| + \frac{(1+\Lambda^2) \, \theta\, (1-\varepsilon)^{-1}}
{(1-\varepsilon)^{-1}\Lambda^{2(1-\a/2)}-1}r^n\, ,
\end{split}
\end{equation}
where we have used that $(1-\varepsilon)^{-1}\Lambda^{2(1-\a/2)}>1$.
Therefore, $|a_n|=\mathcal O(r^n(|a_0|+\theta))$, for $n\to\infty$,
Setting $\rho:=\max(1/\Lambda^2 ,r^2)$, we further write
\begin{equation}
a_{n+1}= \frac{1-e^{\frac{i(1+\Lambda^2) \theta}{\Lambda^{2(n+1)}}}+\Lambda^{-\alpha }a_n}{1+\Lambda^{-\alpha} a_n}=\left(\Lambda^{-\alpha }a_n + \mathcal O(\theta\Lambda^{-2n})\right)(1 + \mathcal O(\Lambda^{-\alpha} a_n))=\Lambda^{- \alpha}a_n+b_n\, ,
\end{equation}
with $b_n:=a_{n+1}-\Lambda^{- \alpha }a_n = \mathcal O(\rho^{n})$, uniformly in $n$ and $\theta \in (0,\theta_0]$.
In this notation, we observe that
\begin{equation}
a_{n}= \Lambda^{-\alpha n}a_0 +\sum_{m=0}^{n-1} \Lambda^{-\alpha m}b_{n-m} 
=\Lambda^{-\alpha n}\left(a_0 +\sum_{m=1}^{n} \Lambda^{\alpha m} b_m 
\right) \, .
\end{equation}
Since $\rho=\max(1/\Lambda^2 ,r^2)<\Lambda^{-2\alpha/2}$, the above sum is convergent for $n\to\infty$. We conclude that 
\begin{equation}
a_n= c_\theta \, \theta^{\a/2} \Lambda^{-2\alpha n /2} + \mathcal O(\rho^n) \,,
\end{equation}
uniformly in $\theta\in (0,\theta_0]$ and with $c_\theta := \theta^{-\a/2}\left(a_0 +\sum_{m=1}^\infty \Lambda^{2\frac {\alpha m}2} b_m\right)$.
Furthermore, since $c_\theta=c_{\theta/\Lambda^2}$ by the definition \eqref{def:ctheta}, we have 
\begin{equation}
\sup_{\theta\in (0,\theta_0]} |c_\theta| = \sup_{\theta\in (\theta_0/\Lambda^2,\theta_0]} |c_\theta|
<\infty \, .
\end{equation}

All in all, we have proved that there exist $\theta_0, C_0>0$ such that, for all $\theta\in[-\theta_0,0) \cup (0,\theta_0]$, 
\begin{equation} \label{prop1}
\lim_{n \to \infty}(1-\varphi_Z(\theta/\Lambda^{2n}))\Lambda^{n\alpha}= c_\theta |\theta|^{\alpha/2} \,,
\end{equation}
with $|c_\theta|<C_0$. Moreover, the convergence is exponentially fast in $n$, uniformly in $[-\theta_0,0) \cup (0,\theta_0]$. To conclude the proof, it remains to prove that $\theta\mapsto c_\theta$ is not identically null. This follows from the next lemma.

\begin{lemma} \label{minolin}
Let $\theta\in (\theta_0/\Lambda^2 ,\theta_0]$. Then, either 
\begin{equation} \label{H1}
	c_\theta=\lim_{n\to\infty}\ell(\theta/\Lambda^{2n})
    =\lim_{n\to\infty}\frac{1-\varphi_Z(\theta/\Lambda^{2n})}{|\theta/\Lambda^{2n}|^{\a/2}}
	\ne 0  \, ,
\end{equation}
or
\begin{equation} \label{H2}
	\lim_{n \to \infty} \frac{1-\varphi_Z(\theta/\Lambda^{2n})}{\theta/\Lambda^{2n}}
	=\tilde{c} i\,,\quad \mbox{with } \tilde{c} := \frac{1+\Lambda^2} {\ds \frac{\Lambda^2 \pup} \pdown-1}>0 \,.
\end{equation}
\end{lemma}

\begin{proof}
For the sake of the notation, let us set 
\begin{equation} \label{def-tildeelln}
\tilde{\ell}_n(\theta):= \frac{1-\varphi_Z(\theta/\Lambda^{2n})}{\theta/\Lambda^{2n}} \,.
\end{equation}
Our first goal will be to prove that there exists $\theta\in (\theta_0/\Lambda^2, \theta_0]$ such that
\begin{equation} \label{degenerate2}
\lim_{n \to \infty} |\tilde{\ell}_n(\theta)|=\infty\, .
\end{equation}

Observing that $\tilde{c}i=-(1+\Lambda^2) i+\frac{\Lambda^2 \pup}{\pdown}\tilde{c}i$, we can rewrite \eqref{1-phi} as 
\begin{equation}
1-\varphi_Z(\theta)-\tilde{c}i\theta=\frac{\pdown(1+(1+\Lambda^2) i\theta-e^{i(1+\Lambda^2) \theta})+\pup[(1-\varphi_Z(\Lambda^2 \theta))(1-\tilde{c}i\theta)-\Lambda^2 \tilde{c}i\theta]}{\pdown+\pup\left(1-\varphi_Z(\Lambda^2 \theta)\right)} \,.
\end{equation}
The above is an identity for $\theta \in \R$. In the limit $\theta \to 0^+$ it implies that
\begin{equation} \label{diffc0}
\frac{1-\varphi_Z(\theta)}\theta-\tilde{c}i 
= \frac{\mathcal O(\theta)+\frac{\Lambda^2 \pup}{\pdown}(\frac{1-\varphi_Z(\Lambda^2 \theta)}{\Lambda^2 \theta}-\tilde{c}i)}{1+ \mathcal O(\theta)} 
= \mathcal O(\theta) + \frac{\Lambda^2\pup}{\pdown}\left(\frac{1-\varphi_Z(\Lambda^2 \theta)}{\Lambda^2 \theta}-\tilde{c}i\right) .
\end{equation}

Let us now restrict to $\theta \in (\theta_0/\Lambda^2 ,\theta_0]$.
Assume first that $\theta$ is such that $\tilde{\ell}_n(\theta)$ does not converge to $\tilde{c}i$, as $n\to\infty$. Then there exists $\delta>0$ such that $\big| \tilde{\ell}_n(\theta) - i\tilde{c} \big|>\delta$ infinitely often in $n$.
In particular, for any $\eta\in \Big(0,\frac{\Lambda^2 \pup} \pdown -1 \Big)$, we can choose $m_0$ such that the error term in the rightmost side of \eqref{diffc0} satisfies $\mathcal O(\theta/\Lambda^{2m})<\eta\delta$ for all $m\ge m_0$, and $\big| \tilde{\ell}_{m_0}(\theta)-\tilde{c} i \big|>\delta$. It follows that
\begin{equation}
	\left|\tilde{\ell}_{m_0+1}(\theta)-\tilde{c}i\right|
	>\left(\frac{\Lambda^2 \pup}{\pdown}-\eta\right)
	\left|\tilde{\ell}_{m_0}(\theta)-\tilde{c}i\right|>\delta \, .
\end{equation}
Proceeding by induction, for $m\ge m_0$ we obtain
\begin{equation}
\left|\tilde{\ell}_{m}(\theta)-\tilde{c}i\right|
>\left(\frac{\Lambda^2 \pup}{\pdown}-\eta\right)^{m-m_0}
\left|\tilde{\ell}_{m_0}(\theta)-\tilde{c}i\right| >
\delta \left(\frac{\Lambda^2 \pup}{\pdown}-\eta\right)^{m-m_0} \, ,
\end{equation}
implying that $\ds \lim_{m\to\infty} \big| \tilde{\ell}_{m}(\theta)-\tilde{c}i \big|=\infty$.

We have thus shown that either $\tilde{\ell}_{m}(\theta)$ converges to $\tilde{c}i$ or its modulus diverges, as $m \to \infty$. Moreover, if for all $\theta\in (\theta_0/\Lambda^2 ,\theta_0]$ \eqref{degenerate2} is false,
then for all $\theta\in (\theta_0/\Lambda^2, \theta_0]$ $\tilde{\ell}_{n}(\theta) \to \tilde{c}i$. If this were the case, the same limit would then occur for all $\theta\in\R^+$. In view of definition \eqref{def-tildeelln}, this would imply that for all $\theta\in\R^+$,
\begin{equation}
    E_0 \!\left[e^{i\theta\Lambda^{-2n}\sum_{k=1}^{\lfloor \Lambda^{2n} \rfloor} Z_k}\right]
	=\varphi_Z(\theta/\Lambda^{2n})^{\lfloor \Lambda^{2n} \rfloor} = \left(1- (\theta/\Lambda^{2n})(i\tilde{c}
    +o(1) \right)^{\Lambda^{2n}} \underset{n\to\infty}{\loongrightarrow} e^{-i\tilde{c}\theta}\,,
\end{equation}
showing that the sequence of positive random variables $\Lambda^{-2n} \sum_{k=1}^{\lfloor \Lambda^{2n} \rfloor} Z_k$ converges in distribution to the negative constant $-\tilde{c}$, which is impossible. 

In conclusion, we have proved that there exists $\theta\in (\theta_0/\Lambda^2,\theta_0]$ such that \eqref{degenerate2} holds, and that for every $\theta\in (\theta_0/\Lambda^2 ,\theta_0]$,
\begin{equation}
\mbox{either} \quad \lim_{n \to \infty}\tilde{\ell}_{n}(\theta)=\tilde{c} i
\quad \mbox{or} \quad
\lim_{n \to \infty}|\tilde{\ell}_{n}(\theta)|=\infty \, .
\end{equation}
From now on, we assume that $\theta \in (\theta_0/\Lambda^2\theta_0]$ is such that 
\begin{equation} \label{toto}
\lim_{n \to \infty}|\tilde{\ell}_{n}(\theta)|=\left| \frac{1-\varphi_Z(\theta/\Lambda^{2n})}{\theta/\Lambda^{2n}}\right| =\infty\, ,
\end{equation}
so that condition \eqref{H2} is not verified, and we aim prove that $c_\theta\ne 0$, as stated by the alternative condition \eqref{H1}.

Assume by contradiction that $c_\theta=0$. It follows from \eqref{prop1} that there exists $\varepsilon>0$ such that, for all $\theta\in (\theta_0/\Lambda^2\theta_0]$, as $n\to\infty$,
\begin{equation} \label{degenerate}
\quad 1-\varphi_Z(\theta/\Lambda^{2n})=\mathcal O \!\left( (\theta/\Lambda^{2n})^{\a/2+\varepsilon} \right).
\end{equation}
Set $\gamma:=\a/2+\varepsilon$. 
Applying \eqref{1-phi} to $\theta/\Lambda^{2n+1}$ gives
\begin{equation} \label{Bn-iter}
\frac{|1-\varphi_Z(\theta/\Lambda^{2(n+1)})|}{|\theta/\Lambda^{2(n+1)}|^\gamma}
= \frac{|1-\varphi_Z(\theta/\Lambda^{2n})|}{|\theta/\Lambda^{2n}|^\gamma}
\Lambda^{2\gamma}\frac{\pup}{\pdown} B_n(\theta)\,,
\end{equation}
where 
\begin{equation}
B_n(\theta) := \frac{\left|1+ \ds \frac{\pdown}{\pup}
\frac{1-e^{i(1+\Lambda^2)\theta/\Lambda^{2(n+1)}}}{1-\varphi_Z(\theta/\Lambda^{2n})}
\right|} 
{\left|1+ \ds \frac{\pup}{\pdown}\left( 1-\varphi_Z(\theta/\Lambda^{2n})\right) \right|} \,.
\end{equation}
Notice that the above numerator tends to 1, by the assumption \eqref{toto}. The denominator also tends to 1, by a property of the characteristic function. We can thus write $B_n(\theta) =: 1-\epsilon_n$, where $\epsilon_n \to 0$, transforming \eqref{Bn-iter} into
\begin{equation}\label{iter}
\frac{|1-\varphi_Z(\theta/\Lambda^{2(n+1)})|}{|\theta/\Lambda^{2(n+1)}|^\gamma}
= \frac{|1-\varphi_Z(\theta/\Lambda^{2n})|}{|\theta/\Lambda^{2n}|^\gamma}
\Lambda^{2\gamma}\frac{\pup}{\pdown}(1-\epsilon_n)\, .
\end{equation}
Now, let $\delta\in \big(1,\Lambda^{2\gamma}\frac \pup \pdown \big)$ (this is a non-empty set by definition of $\gamma$) and $n_0\in\mathbb{N}$ be such that $\Lambda^{2\gamma}\frac \pup \pdown (1-\epsilon_n) \ge \delta$,
for all $n\ge n_0$. For such values of $n$, iterating \eqref{iter} gives 
\begin{equation}
\left|\frac{1-\varphi_Z(\theta/\Lambda^{2n})}{|\theta/\Lambda^{2n}|^\gamma}\right|\ge \delta^{n-n_0} 
\left|\frac{1-\varphi_Z(\theta/\Lambda^{2n_0})}{|\theta/\Lambda^{2n_0}|^\gamma}\right| ,
\end{equation}
which implies that
\begin{equation}
\lim_{n \to \infty}
\left|\frac{1-\varphi_Z(\theta/\Lambda^{2n})}{|\theta/\Lambda^{2n}|^\gamma}\right|=\infty\, .
\end{equation}
This contradicts \eqref{degenerate}, and thus the fact that $c_\theta=0$, ending the proof of Lemma \ref{minolin}.
\end{proof}

\subsection{Final arguments} 
\label{sec:proofCase2}
\begin{proof}[Proof of Theorem~\ref{THMcase2}]
Let us first consider the case of the ``overscaling'' $b_n\gg n^{1/\a}$. Applying the the Markov inequality under the conditioning to $L = (L_k)_{k\in\mathbb{N}}$, we get that, for all $\delta>0$,
\begin{equation}
\P_0(| M_n/b_n|>\delta) = E_0 \big[ \P_0(| M_n/b_n|>\delta \,|\, L)  \big] \le 
E_0 \!\left[ \min \!\left( 1, \frac{V_n}{b_n^2\delta^2} \right) \right] .
\end{equation}
It follows from Corollary \ref{biggamma} that $V_n/b_n^2$, and thus $\min\!\Big( 1,\frac{V_n} {b_n^2\delta^2} \Big)$, converges in $P_0$-probability to 0. So the above l.h.s.\ vanishes for $n\to\infty$. We conclude that $M_n/b_n$ converges to 0 in probability, therefore in distribution, relative to $\P_0$. This is equivalent to \eqref{upper} via Lemma \ref{lemZwei}.

It remains to prove the second part of the theorem. In this case, $b_n\ll n^{1/\alpha}$. Conditionally to $L$, $M_n$ is a centered Gaussian with variance $V_n$. So, for all $x>0$,
\begin{equation}
\P_\nu (|M_n|>x) = E_\nu\! \left[ G(x/\sqrt{V_n}) \right] ,
\end{equation}
where $G(x) = \P(|\xi_0|>x)$. Setting $n_j:=\lfloor \Lambda^{\alpha j}\rfloor$, it follows from \eqref{conv:Z}, combined with Remark \ref{rem-measure}, that $V_{n_j}/n_j^{2/\a}$ converges in distribution to $\widetilde Z$, w.r.t.\ $P_{\nu}$, as $j \to \infty$. Moreover, $\widetilde Z$ has no atom in 0 (because $\varphi_{\widetilde Z}(\theta) \to 0$, as $\theta \to +\infty$). As a consequence, $b_{n_j}^2/V_{n_j}$ converges to 0 in distribution (and so in probability) w.r.t.\ $P_\nu$, whence, for all $r>0$, 
\begin{equation}
\P_\nu (|M_{n_j}|/b_{n_j}>r) = E_\nu\!\left[ G\!\left(r \, b_{n_j}/\sqrt{V_{n_j}} \,\right)\right],
\end{equation}
which tends to 1 as $j\to \infty$, thus providing \eqref{lower}.
\end{proof}

\begin{proof}[Proof of Theorem~\ref{THMcase2-convergence}]
Using the assumption that $\xi_k$ are standard Gaussians, we obtain
\begin{equation}
\begin{split}
\E_0 \!\left[\left.\exp\left(\frac{i\theta M_n}{n^{1/\a}} \right)\right| L\right]
&= \E_0 \!\left[\exp\left(\left.\frac{i\theta}{n^{1/\a}} \sum_{j=0}^{n-1} \Lambda ^{L_j}\xi_j \right) \right| L\right] \\
&= \prod_{j=0}^{n-1} \exp\left(-\frac{\theta^2}{2n^{2/\a}} \Lambda^{2L_j} \right) .
\end{split}
\end{equation}
As a consequence, by definition \eqref{def:Vn},
\begin{equation} \label{conditionalcarfunc}
\E_0 \!\left[\left. \exp\left(\frac{i\theta M_n}{n^{1/\a}} \right) \right|L \right]
 = \exp \!\left(-\frac{\theta^2V_n}{2n^{2/\a}}\right) ,
\end{equation}
and taking the average, we get
\begin{equation} \label{carfunc}
\E_0 \!\left[\exp\left(\frac{i\theta M_n}{n^{1/\a}}\right)\right]
= E_0 \!\left[\exp\left(-\frac{\theta^2V_n}{2n^{2/\a}}\right)\right] .
\end{equation}
Let us first focus on item (\ref{item1}).
We assume that the function $\R^+\ni\theta \mapsto c_{\theta}$ is not constant.
It follows from Corollary~\ref{CoroCase3} that, when $n\to\infty$, $V_{\tau_0^{(n)}}/n^{2/\a}$ does not converge in distribution w.r.t.~$P_0$, hence neither w.r.t.~$P_{\mu}$ (Lemma~\ref{lemZwei}).

It remains to show that the fact that $V_{\tau_0^{(n)}}/n^{2/\a}$ does not converge in distribution w.r.t.\ $P_{\mu}$ implies that $M_n/n^{1/\a}$ does not converge in distribution w.r.t.\ $\P_{\mu}$. To this end we prove the contrapositive. 
Assume that $M_n/n^{1/\a}$ converges in distribution to a random variable with characteristic function $\varphi$, and let us deduce that $V_{\tau_0^{(n)}}/n^{2/\a}$ also converges in distribution. 
Notice that it follows from \eqref{carfunc} and Lemma \ref{lemZwei} that the convergence in distribution of $M_n/n^{1/\a}$ to a random variable with characteristic function $\varphi$ is equivalent to the convergence in distribution of $V_n/n^{2/\a}$ to a random variable $\widetilde{V}$ with Laplace transform $\vartheta\mapsto \varphi(\sqrt{2\vartheta})$. We claim that this implies that $V_{\tau_0^{(n)}}/n^{2/\a}$ converges in distribution to $(E_0[\tau_0])^{2/\a} \, \widetilde{V}$.

To prove the claim, we start by noticing that 
since the tail distribution of $\tau_0$ decays exponentially, we have
\begin{equation} \label{omega}
\lim_{n \to \infty} P_0\!\left(  n E_0[\tau_0]-n^{3/4}\le \tau_0^{(n)}\le n E_0[\tau_0]+n^{3/4}\right)=1 \, .
\end{equation}
Moreover, since $n\mapsto V_n$ is increasing, the following inequalities 
hold true on the set $\Omega_n:=\{n E_0[\tau_0]-n^{3/4}\le \tau_0^{(n)}\le n E_0[\tau_0]+n^{3/4}\}$,
\begin{equation}
n^{-2/\a} \, V_{n E_0[\tau_0]-n^{3/4}}\le n^{-2/\a} \, V_{\tau_0^{(n)}}\le n^{-2/\a} \, V_{n E_0[\tau_0]+n^{3/4}}\,.
\end{equation}
On the other hand, under $P_{\mu}$, $n^{-2/\a} \big(V_{n E_0[\tau_0]+n^{3/4}} -V_{n E_0[\tau_0]-n^{3/4}} \big)$
has the same distribution as $n^{-2/\a} \, V_{2n^{3/4}}$,
which converges in distribution (equivalently, in probability) to 0.
Thus we can write
\begin{equation}
n^{-2/\a} \, V_{\tau_0^{(n)}}=A_n+B_n+C_n\, ,
\end{equation}
where
\begin{equation}
\begin{split}
A_n &:= n^{-2/\a} \, V_{n E_0[\tau_0]-n^{3/4}} \, \mathbf 1_{\Omega_n}, \\
B_n &:= n^{-2/\a} \, V_{\tau_0^{(n)}} \, \mathbf 1_{\Omega_n^c} , \\
C_n &:= n^{-2/\a} \big( V_{\tau_0^{(n)}}-V_{n E_0[\tau_0]-n^{3/4}} \big) 
\mathbf 1_{\Omega_n} .
\end{split}
\end{equation}
Since $C_n\in \big[0, n^{-2/\a} \big(V_{n E_0[\tau_0]+n^{3/4}}-V_{n E_0[\tau_0]-n^{3/4}} \big) \big]$, we conclude that $C_n\to 0$ in distribution
(and in probability).
By the definition of $\Omega_n$ and \eqref{omega}, 
$B_n$ also converges to $0$, while $A_n$ converges in distribution to $(E_0[\tau_0])^{2/\a} \, \widetilde{V}$. These three convergences, together with Slutzky's Theorem, end the proof of the contrapositive and thus conclude the proof of item (\ref{item1}). 

\medskip

We finally provide the proof of item (\ref{item2}).
As before, it suffices to restrict our attention to $\theta>0$.
We assume that $c_\theta$ is equal to a constant $c$, for all $\theta>0$. 
Proposition \ref{asymptphi} ensures that $c\ne 0$. Moreover, there exist $\theta_0, C, a>0$ such that
\begin{equation}
\sup_{\theta\in (\theta_0/\Lambda^{2},\theta_0]} \left| (1-\varphi_Z(\theta\Lambda^{-2n}))\Lambda^{n\alpha}-c\,\theta^{\a/2}\right|
\le Ce^{-an} \,.
\end{equation}
This implies that
\begin{equation} \label{22}
\begin{split}
\sup_{\theta\in (\theta_0/\Lambda^2,\theta_0]} \left| 
\frac{1-\varphi_Z(\theta\Lambda^{-2n})}{(\theta/\Lambda^{2n})^{\alpha/2}}-c \,\right|
&\le \sup_{\theta\in (\theta_0/\Lambda^2,\theta_0]} \theta^{-\a/2}\left| (1-\varphi_Z(\theta\Lambda^{-2n}))\Lambda^{n\alpha}-c\, \theta^{\a/2}\right| \\
&\le (\theta_0/\Lambda^2)^{-\a/2} \, Ce^{-an} =: C' e^{-an} \,.
\end{split}
\end{equation}
Let $n\in\N$ and $\theta>0$ be such that 
\begin{equation} \label{23}
\theta_0/\Lambda^{2(n+1)} < \theta \le \theta_0/\Lambda^{2n} \,.
\end{equation}
Then $s:=\theta\Lambda^{2n}$ satisfies $s\in (\theta_0 / \Lambda^2,\theta_0]$ and so
\begin{equation}
\left| \frac{1-\varphi_Z(\theta)}{\theta^{\alpha/2}}-c\, \right| 
= \left| \frac{1-\varphi_Z(s/\Lambda^{2n})}{(s/\Lambda^{2n})^{\alpha/2}}
-c\, \right| \le C' e^{-an} = 
C'\left(\Lambda^{-2n}\right)^{\frac a{2\log \Lambda}} \,,
\end{equation}
by \eqref{22}. Combining this with \eqref{23}, we obtain
\begin{equation}
\left|\frac{1-\varphi_Z(\theta)}{\theta^{\alpha/2}}-c\, \right|
\le C'\left(\frac {\Lambda^2}{\theta_0}\theta_0\Lambda^{-2(n+1)}\right)^{\frac a{2\log \Lambda}} < 
C'\left(\frac {\Lambda^2}{\theta_0} \theta\right)^{\frac a{2\log \Lambda}} \, .
\end{equation}
We conclude that 
\begin{equation} \label{eq:asymptphi}
1-\varphi_Z(\theta)\sim c\, \theta^{\a/2}\,, \quad \mbox{as }\theta\to 0^+\, .
\end{equation}

Since $V_{\tau_0^{(n)}}-V_{\tau_0^{(n-1)}}$ ($n\in\N$), are i.i.d.\  variables with common characteristic function $\varphi_Z$ satisfying \eqref{eq:asymptphi}, $\big( V_{\tau_0^{(\lfloor n t\rfloor)}}/n^{2/\a} \big)_{t\ge 0}$ converges in distribution, w.r.t.~$P_0$, to an $(\alpha/2)$-stable process $(\widetilde Z_t)_{t\ge 0}$ with independent increments, such that the characteristic function of $\widetilde Z_1$, restricted to $\theta>0$, is $\theta\mapsto e^{-c\,\theta^{\a/2}}$. Since $\widetilde Z_1$ is a positive $(\a/2)$-stable variable, its characteristic function must have the form 
\begin{equation}
\theta\mapsto e^{-c'|\theta|^{\a/2}(1-i\,\sgn(\theta)\tan(\pi\a/4))} ,
\end{equation}
with $c'>0$ \cite[Chap.~2]{IL}. Comparing the two expressions we get
\begin{equation}
c'=\frac c{1-i\tan(\pi\a/4)}=c\cos(\pi\a/4) \, e^{i\pi\a/4} \, .
\end{equation} 
Therefore, the Laplace transform of $\widetilde Z_1$ is given by
\begin{equation}
E_0\!\left[e^{-\vartheta \widetilde Z_1}\right] =
e^{-c\, \vartheta^{\a/2} \, e^{i\pi\a/4}} =
e^{-\frac{c' \vartheta^{\a/2}}{\cos(\pi \a/4)}} \quad (\vartheta \ge 0)\, .
\end{equation}

To deal with the change of time $n\mapsto\t_0^{(n)}$, recall that
\begin{equation} \label{sand}
V_{\tau_0^{(N_{\lfloor nt\rfloor}(0))}}
\le V_{\lfloor nt\rfloor}
< V_{\tau_0^{(1+N_{\lfloor nt\rfloor}(0))}} \, ,
\end{equation}
where $N_n(0)$ is the local time of $L$ at 0 up to time $n$.
We know that $N_n(0)/n$ converges $P_0$-almost surely to $\mu_0$, and that, as a consequence, the processes $\left( N_{\lfloor nt\rfloor}(0)/n \right)_{t\ge 0}$ converges $P_0$-almost surely, uniformly on every compact, to the deterministic limit $\mu_0 \, \mathrm{id}$. 
 
Combining the convergence results for $\big( \big( V_{\tau_0^{(\lfloor nt\rfloor)}}/n^{2/\a} \big)_{t\ge 0} \big)_{n\ge 1}$ and for the sequence of increasing processes $\big( \big( N_{\lfloor nt\rfloor}(0)/n \big)_{t\ge 0} \big)_{n\ge 1}$, we conclude from \eqref{sand}, applying \cite[Thm.~13.2.3]{WW1}, that $\big( V_{\lfloor nt\rfloor}/n^{2/\a} \big)_{t\ge0}$ converges to $\big( \widetilde Z_{\mu_0 t} \big)_{t\ge 0}$, both in distribution for the Skorokhod $M_1$-topology and in the sense of finite-dimensional distributions.
Therefore, for any $t_1, \ldots, t_p \in [0,+\infty)$ with $0=t_0 \le t_1 \le \ldots \le t_p$ and $u_1, \ldots, u_p \in [0,+\infty)$, we have
\begin{equation}
\lim_{n \to \infty} E_0 \!\left[\exp \!\left( -\sum_{j=1}^p u_j \,
\frac{V_{\lfloor nt_j\rfloor} - V_{\lfloor nt_{j-1}\rfloor}}{n^{2/\a}} 
\right)\right] =\prod_{j=1}^p e^{-c\, \mu_0 \, e^{i\pi\a/4} \, (t_j-t_{j-1})\, u_j^{\a/2}} .
\end{equation}
It follows from \eqref{conditionalcarfunc} that, for all $\theta_1, \ldots, \theta_p \in\R$,
\begin{equation}
\E_0 \!\left[ \exp \!\left( \sum_{j=1}^p i\theta_j \, 
\frac{M_{\lfloor nt_j\rfloor} - M_{\lfloor nt_{j-1}\rfloor}}
{n^{2/\a}} \right) \right] =
E_0 \!\left[\exp \!\left( -\frac 12\sum_{j=1}^p \theta_j^2 \,
\frac{V_{\lfloor nt_j\rfloor} - V_{\lfloor nt_{j-1}\rfloor}}
{n^{2/\a}} \right) \right] .
\end{equation}

In conclusion,
\begin{equation}
\lim_{n \to \infty} \E_0 \!\left[\exp\left(
\sum_{j=1}^p i\theta_j \, \frac{M_{\lfloor nt_j\rfloor} - 
M_{\lfloor nt_{j-1}\rfloor}}{n^{2/\a}} \right) \right] =
\prod_{j=1}^p e^{-\frac c2\, \mu_0 \, e^{i\pi\a/4} \, (t_j-t_{j-1})|\theta_j|^\alpha } \, ,
\end{equation}
which corresponds to the joint characteristic function of the increments $\mathcal{Y}_{t_j} - \mathcal{Y}_{t_{j-1}}$ ($j=1, \ldots, p$) of a symmetric $\alpha$-stable process $( \mathcal{Y}_{t} )_{t\ge 0}$, as claimed, with 
\begin{equation} \label{c-tilde}
\widetilde c:=\frac{c\, \mu_0}2 \, e^{i\pi\a/4}\, .
\end{equation}
\end{proof}

\section{\texorpdfstring{$Z$}{Z} is not in the domain of attraction of a stable random variable}
\label{sec-cap}

Here we establish a number of mathematical propositions that can be used in
conjunction with certified numerical computations
to prove that $Z$ is not in the domain of attraction of a stable random
variable. (By `certified numerical computation' we mean a computation
consisting of a finite number of operations whose numerical result is
endowed with a rigorously proved bound for the maximum error. The proof
of the bound may be provided by a human or a computer working in
interval arithmetic --- hence with absolute precision.)

By Proposition \ref{CoroCase3}, it will suffice to
establish that $\theta \mapsto c_\theta$ (recall the definition from
Proposition \ref{asymptphi}) is not a constant \fn. We believe the
following:

\begin{conjecture} \label{conjecture}
  For all $\Lambda>1$ and $\a \in (0,2)$, there exist 
  $\theta_1, \theta_2 > 0$ such that $c_{\theta_1} \ne 
  c_{\theta_2}$.
\end{conjecture}

In fact, up to certifying our numerics (which will be abundantly
accurate nonetheless), we prove the following.

\begin{proposition} \label{prop-conj}
  There are examples of $\Lambda$, $\a, \theta_1, \theta_2$ for which
  $c_{\theta_1} \ne c_{\theta_2}$.
\end{proposition}

We start by introducing the convenient parameter $\anot := \a/2$. 
In order to study $c_\theta = \lim_{n\to\infty} \ell(\Lambda^{-2n} \theta)$, we apply
(\ref{fixed-ell0}) to $\theta/\Lambda^2$ in place of $\theta$ and rewrite the resulting
equation in terms of a fractional linear equation:
\begin{equation} 
  \left[ \begin{array}{c}
    c \ell(\Lambda^{-2} \theta) \\ c
  \end{array} \right] = 
   \left[ \begin{array}{cc}
    1 & \frac{1-e^{i(1+\Lambda^2) \theta/\Lambda^2 }} {(\theta/\Lambda^2 )^\anot} \\[6pt]
    (\theta/\Lambda^2)^\anot & 1 
  \end{array} \right] 
  \left[ \begin{array}{c}
    \ell(\theta) \\ 1
  \end{array} \right] 
  =: A(\theta)
  \left[ \begin{array}{c}
    \ell(\theta) \\ 1
  \end{array} \right] ,
\end{equation}
where $c \ne 0$ depends on $\theta$ and $\ell(\theta)$. If follows that, for some (other)
$c\ne 0$,
\begin{equation} \label{fundamental}
  c \left[ \begin{array}{c}
    \ell(\Lambda^{-2n} \theta) \\ 1
  \end{array} \right] 
  =
  A(\Lambda^{-2(n-1)} \theta) \, A(\Lambda^{-2(n-2)} \theta) \cdots A(\theta) 
  \left[ \begin{array}{c}
    \ell(\theta) \\ 1
  \end{array} \right] .
\end{equation}
Let us denote
\begin{equation} \label{atn}
  A(\theta) = T(\theta) + N(\theta) :=
  \left[ \begin{array}{cc}
    1 & 0 \\[6pt]
    (\theta/\Lambda^2)^\anot & 1 
  \end{array} \right]
  +
  \left[ \begin{array}{cc}
    0 & \frac{1-e^{i(1+\Lambda^2) \theta/\Lambda^2 }} {(\theta/\Lambda^2 )^\anot} \\[6pt]
    0 & 0 
  \end{array} \right] .
\end{equation}
Given the simple group property verified by the lower-triangular matrices $T$, 
it would be convenient to replace all the $A$-matrices in (\ref{fundamental}) with 
the corresponding $T$-matrices. We do so and give an estimate of the error.

\begin{lemma} \label{lem-est}
  For all $\theta>0$ and all $n \in \N$,
  \begin{displaymath}
  \left\| A(\Lambda^{-2(n-1)} \theta) \, A(\Lambda^{-2(n-2)} \theta) \cdots A(\theta) - 
  T(\Lambda^{-2(n-1)} \theta) \, T(\Lambda^{-2(n-2)} \theta) \cdots T(\theta) \right\| \le \Psi(\theta),
  \end{displaymath}
  where $\| \cdot \|$ denotes the Euclidean operator norm for $2 \times 2$ 
  complex matrices and
  \begin{equation} \label{psi}
    \Psi(\theta) := \prod_{j=1}^\infty \left[ 1 + (\Lambda^{-2j}\theta)^\anot \right]
    \cdot \sum_{k=1}^\infty \frac{ (1+\Lambda^2) ^k (\Lambda^{-2k} \theta)^{k - \anot} } 
    { (1 - \Lambda^{-2\anot})^{k-1} \prod_{l=1}^k (1 - \Lambda^{-2(l-\anot)}) } .
  \end{equation}
  Clearly, also, 
  \begin{equation} \label{prod-of-t}
    T(\Lambda^{-2(n-1)} \theta) \, T(\Lambda^{-2(n-2)} \theta) \cdots T(\theta) =
    \left[ \begin{array}{cc}
      1 & 0 \\[4pt]
      \frac{1-\Lambda^{-2\anot n}} {1 - \Lambda^{-2\anot}} (\theta/\Lambda^2 )^\anot & 1 
    \end{array} \right] .
  \end{equation}
\end{lemma}

\begin{remark} \label{rk-psiq}
  Though the error estimate $\Psi(\theta)$ looks rather cumbersome, it may be
  further bounded above, for 
  \begin{equation} 
    0 < \frac{(1+\Lambda^2) \theta/\Lambda^2 } {(1 - \Lambda^{-2\anot}) 
    (1 - \Lambda^{-2(1-\anot)})} < 1,
  \end{equation}
  by
  \begin{equation} \label{psi1}
    \Psi(\theta) < \Psi_1(\theta) := \exp\! \left( \frac{(\theta/\Lambda^2 )^\anot} 
    {1 - \Lambda^{-2\anot}} \right) \frac{ \ds \frac{(1+\Lambda^2) 
    (\theta/\Lambda^2 )^{1-\anot}} {1 - \Lambda^{-2(1-\anot)}} }
    { \ds 1 - \frac{(1+\Lambda^2) \theta/\Lambda^2 } {(1 - \Lambda^{-2\anot}) 
    (1 - \Lambda^{-2(1-\anot)})}  }.
  \end{equation}
  This is readily seen by operating the following bounds on $\Psi(\theta)$:
  \begin{equation} 
  \begin{split}
    & 1 + (\Lambda^{-2j}\theta)^\anot < \exp\! \left( \theta^\anot \Lambda^{-2\anot j} \right) ; \\
    & (\Lambda^{-2k} \theta)^{k - \anot} \le (\theta/\Lambda^2 )^{k - \anot} ; \\
    & 1 - \Lambda^{-2(l-\anot)} \ge 1 - \Lambda^{-2(1-\anot)},
  \end{split}
  \end{equation}
  and then summing a geometric series in $k$. In fact, the bound (\ref{psi1}) 
  is very generous. Tighter bounds on $\Psi(t)$ can be produced by leaving out 
  the first terms of the sum in $k$ and estimating the others by means of a 
  geometric series, in analogy to what was done for $\Psi_1(\theta)$. So, for all
  $q \ge 2$ and 
  \begin{equation} \label{base2}
    0 < \frac{(1+\Lambda^2) \Lambda^{-2q} \theta} {(1 - \Lambda^{-2\anot}) 
    (1 - \Lambda^{-2(q-\anot)})} < 1,
  \end{equation}
  we have:
  \begin{equation} \label{psiq}
  \begin{split}
    & \Psi(\theta) < \Psi_q(\theta) := \exp\! \left( \frac{(\theta/\Lambda^2 )^\anot} 
    {1 - \Lambda^{-2\anot}} \right) \times \\
    &\times \left(  \sum_{k=1}^{q-1} \frac{ (1+\Lambda^2)^k \,
    (\Lambda^{-2k} \theta)^{k - \anot} } 
    { (1 - \Lambda^{-2\anot})^{k-1} \prod_{l=1}^k (1 - \Lambda^{-2(l-\anot)}) } +
    \frac 1 { \prod_{l=1}^{q-1} (1 - \Lambda^{-2(l-\anot)} ) } \,
    \frac{ \ds \frac{(1+\Lambda^2)^q \, (\Lambda^{-2q} \theta)^{q-\anot} } 
    { (1 - \Lambda^{-2\anot})^{q-1} \, (1 - \Lambda^{-2(q-\anot)})} }
    { \ds 1 - \frac{(1+\Lambda^2) \, \Lambda^{-2q} \theta} {(1 - \Lambda^{-2\anot}) 
    (1 - \Lambda^{-2(q-\anot)})} } \right) .
  \end{split}
  \end{equation}
  The exponential factor in the above r.h.s.\ is derived in the same way as the
  corresponding term in (\ref{psi1}). The second term in parentheses is an 
  upper bound for the sum on $k$ from $q$ to $\infty$ of the summands in
  (\ref{psi}), having used that, for $k\ge q$,
   \begin{equation} 
  \begin{split}
    & (\Lambda^{-2k} \theta)^{k - \anot} \le (\Lambda^{-2q} \theta)^{k - \anot} ; \\
    & \prod_{l=1}^k \left( 1 - \Lambda^{-2(l-\anot)} \right) \ge 
    \left( 1 - \Lambda^{-2(q-\anot)} \right)^{k-q+1}
    \prod_{l=1}^{q-1} \left( 1 - \Lambda^{-2(l-\anot)} \right) .
  \end{split}
  \end{equation}
\end{remark}

\proofof{Lemma \ref{lem-est}} We shall make extensive use of the 
properties of the lower-triangular matrices $T$ and the nilpotent matrices
$N$. In fact, for $z \in \C$, set
\begin{equation} 
  T_z :=
  \left[ \begin{array}{cc}
    1 & 0 \\
    z & 1 
  \end{array} \right],
  \qquad
  N_z :=
  \left[ \begin{array}{cc}
    0 & z \\
    0 & 0 
  \end{array} \right].
\end{equation}
Clearly,
\begin{equation} \label{elem-prods}
  T_{z_1}  T_{z_2} = T_{z_1 + z_2}; \qquad
  N_{z_1}  N_{z_2} = 0; \qquad
  N_{z_1}  T_{z_2} N_{z_3} = N_{z_1 z_2 z_3}\, .
\end{equation}
The first of the above identities readily implies (\ref{prod-of-t}). Also, it is easy to
see that  $1\le \Vert T_z\Vert\le 1+|z|$ and $\Vert N_z\Vert=|z|$.

The expression $A(\Lambda^{-2(n-1)} \theta) \cdots A(\theta) - T(\Lambda^{-2(n-1)} \theta) 
\cdots T(\theta)$ results in the sum of $2^n-1$ products of $n$ matrices, each 
containing $k$ $N$-matrices and $(n-k)$ $T$-matrices, for $k = 1, \ldots n$. 
We describe how to upper-bound (in norm) the expressions for $k = 1, 2, 3$. 
The bound for general $k$ will then be apparent. Let us first set
\begin{equation}
  \phi(\theta) := \frac{1-e^{i(1+\Lambda^2) \theta/\Lambda^2 }} {(\theta/\Lambda^2 )^\anot},
\end{equation}
cf.\ (\ref{atn}). Clearly, $|\phi(\theta)| \le (1+\Lambda^2)  (\theta/\Lambda^2 )^{1-\anot}$.

\medskip

\paragraph{\underline{Case $k=1$}} For $j = 0, 1, \ldots, n-1$, 
$\| T(\Lambda^{-2j} \theta) \|
 \le 1 + (\Lambda^{-2(j+1)} \theta)^\anot$. We bound the norm of the product of the 
 $n-1$ $T$-matrices with the product of their norms, which is in turn less than
\begin{equation} \label{est-10}
  \prod_{j=0}^{n-1} \left[ 1 + (\Lambda^{-2(j+1)} \theta)^\anot \right] < 
  \prod_{j=1}^\infty \left[ 1 + (\Lambda^{-2j}\theta)^\anot \right] =: K.
\end{equation}
As for the remaining $N$-matrix, $\|N(\theta)\| \le |\phi(\theta)|$, so the sum of the
$n$ matrix products is less than
\begin{equation} 
  K \sum_{j_1=0}^{n-1} |\phi (\Lambda^{-2j_1} \theta)| \le K (1+\Lambda^2)  
  \sum_{j_1=0}^\infty (\Lambda^{-2(j_1 + 1)} \theta)^{1-\anot} = K (1+\Lambda^2) 
  \theta^{1-\anot} \frac{\Lambda^{-2(1 - \anot)}} {1 - \Lambda^{-2(1 - \anot)}}.
\end{equation}
A useful way to rewrite this bound, as we shall see later, is 
\begin{equation} 
  K \frac{ (1+\Lambda^2) (\Lambda^{-2} \theta)^{1-\anot} } 
  { 1 - \Lambda^{-2(1-\anot)} }.
\end{equation}

\medskip

\paragraph{\underline{Case $k=2$}} We have $n(n-1)/2$ products 
of the type
\begin{equation} \label{est-20}
\begin{split}
  & T(\Lambda^{-2(n-1)} \theta) \cdots T(\Lambda^{-2(j_2+1)} \theta) N(\Lambda^{-2j_2} \theta)
  T(\Lambda^{-2(j_2-1)} \theta) \cdots \times \\
  &\qquad \times \cdots T(\Lambda^{-2(j_1+1)} \theta) N(\Lambda^{-2j_1} \theta) 
  T(\Lambda^{-2(j_1-1)} \theta) \cdots T(\theta),
\end{split}
\end{equation}
for $0 \le j_1 < j_2 \le n$. In fact, by the second identity of (\ref{elem-prods}), 
the terms with $j_2 = j_1+1$ are null, so we can restrict the range of $j_1, j_2$ 
to $0 \le j_1 \le n-1$, $j_1 + 2 \le j_2 \le n$.

By the first and the third of (\ref{elem-prods}),
\begin{equation} \label{est-30}
\begin{split}
  &\left\| N(\Lambda^{-2j_2} \theta) T(\Lambda^{-2(j_2-1)} \theta) \cdots 
  T(\Lambda^{-2(j_1+1)} \theta) N(\Lambda^{-2j_1} \theta) \right\| \\
  &= \left\| N_{ \phi(\Lambda^{-2j_1} \theta) \left( \sum_{m = j_1+1}^{j_2-1} 
  (\Lambda^{-2(m+1)} \theta)^\anot \right) \phi(\Lambda^{-2j_1} \theta) } \right\| \\ 
  &= |\phi(\Lambda^{-2j_1} \theta)| \left( \sum_{m = j_1+1}^{j_2-1} 
  (\Lambda^{-2(m+1)} \theta)^\anot \right) |\phi(\Lambda^{-2j_1} \theta)| .
\end{split}
\end{equation}
On the other hand, 
\begin{equation} 
  \sum_{m = j_1+1}^{j_2-1} (\Lambda^{-2(m+1)} \theta)^\anot \le \theta^\anot \!\!
  \sum_{m = j_1+2}^\infty (\Lambda^{-2\anot})^m = 
  \frac{ (\Lambda^{-2\anot})^{j_1+2} \, \theta^\anot } { 1 - \Lambda^{-2\anot} }.
\end{equation}
The contribution of the remaining $T$-matrices is again (generously) estimated 
by the constant $K$ defined in (\ref{est-10}).

So the norm of the sum of all the terms of type (\ref{est-20}) is less than
\begin{equation} 
\begin{split}
  &K \sum_{j_1=0}^\infty \, \sum_{j_2=j_1+2}^\infty (1+\Lambda^2) 
  (\Lambda^{-2(j_2+1)} \theta)^{1-\anot} \, \frac{ (\Lambda^{-2\anot})^{j_1+2} \, 
  \theta^\anot } {1 - \Lambda^{-2\anot} } \, (1+\Lambda^2) 
  (\Lambda^{-2(j_1+1)} \theta)^{1-\anot} \\
  &= K\,  \frac{ (1+\Lambda^2) ^2 (\Lambda^{-4} \theta)^{2-\anot} }
  { (1 - \Lambda^{-2\anot}) (1 - \Lambda^{-2(1-\anot)}) 
  (1 - \Lambda^{-2(2-\anot)}) } .
\end{split}
\end{equation}

\medskip

\paragraph{\underline{Case $k=3$}} First of all, observe that (\ref{elem-prods}) 
implies 
\begin{equation} \label{elem-prod2}
  N_{z_1}  T_{z_2} N_{z_3} T_{z_4} N_{z_5} = N_{z_1 z_2 \cdots z_5}
\end{equation}
and all analogous identities for the product of an odd number of alternating 
$N$- and $T$-matrices.

In analogy to (\ref{est-20})-(\ref{est-30}), all the $\binom n 3$ products
with 3 $T$-matrices contain the subproduct
\begin{equation} 
  N(\Lambda^{-2j_3} \theta) T(\Lambda^{-2(j_3-1)} \theta) 
  \cdots T(\Lambda^{-2(j_2+1)} \theta)
  N(\Lambda^{-2j_2} \theta) T(\Lambda^{-2(j_2-1)} \theta) 
  \cdots T(\Lambda^{-2(j_1+1)} \theta) 
  N(\Lambda^{-2j_1} \theta) ,
\end{equation}
whose norm is estimated in the same way as (\ref{est-30}) (as justified
before, one discards the cases $j_2=j_1+1$, $j_3=j_2+1$). Once
again, the contribution of all the remaining $T$-matrices is estimated
by $K$. Thus the sum of all the terms with $k=3$ is bounded in norm by
\begin{equation} 
\begin{split}
  &K \sum_{j_1=0}^\infty \, \sum_{j_2=j_1+2}^\infty \, \sum_{j_3=j_2+2}^\infty 
  (1+\Lambda^2)  (\Lambda^{-2(j_3+1)} \theta)^{1-\anot} \, 
  \frac{ (\Lambda^{-2\anot})^{j_2+2} \, \theta^\anot } {1 - \Lambda^{-2\anot} } \times \\
  &\qquad \times (1+\Lambda^2)  (\Lambda^{-2(j_2+1)} \theta)^{1-\anot} \, 
  \frac{ (\Lambda^{-2\anot})^{j_1+2} \, \theta^\anot } {1 - \Lambda^{-2\anot} } \, 
  (1+\Lambda^2)  (\Lambda^{-2(j_1+1)} \theta)^{1-\anot} = \\
  &= K \, \frac{ (1+\Lambda^2) ^3 (\Lambda^{-6} \theta)^{3-\anot} }
  { (1 - \Lambda^{-2\anot})^2 (1 - \Lambda^{-2(1-\anot)}) 
  (1 - \Lambda^{-2(2-\anot)}) (1 - \Lambda^{-2(3-\anot)}) } .
\end{split}
\end{equation}

Applying the above arguments to the case of a general $k$ (assuming
$n \ge k$) proves Lemma \ref{lem-est}.
\qed

\medskip

Here's how to use Lemma \ref{lem-est} to give a computer-assisted proof of 
Proposition \ref{prop-conj} and make Conjecture \ref{conjecture} quite believable. 

Suppose that, for two relatively small values $\theta_1, \theta_2 > 0$, one is able to
give relatively tight, certified (hence, rigorous), positive upper and lower bounds 
for $|\ell(\theta_j)|$:
\begin{equation} \label{cert-est}
  0 < B_j^- \le |\ell(\theta_j)| \le B_j^+,
\end{equation}
for $j=1,2$. It follows from (\ref{fundamental}) and Lemma \ref{lem-est} that
\begin{equation} 
  \frac{\ell(\Lambda^{-2n} \theta)}  {\ell(\theta)} = \frac{ 1 + r_{11}^{(n)}(\theta) + \ell(\theta)^{-1} 
  r_{12}^{(n)}(\theta) } { \ell(\theta) \frac{1-\Lambda^{-2\anot n}} {1 - \Lambda^{-2\anot}} 
  (\theta/\Lambda^2 )^\anot + \ell(\theta) r_{21}^{(n)}(\theta) +1 + r_{22}^{(n)}(\theta) },
\end{equation}
where $r_{kl}^{(n)}(\theta)$ ($k,l = 1,2$) are the entries of the matrix 
\begin{equation} 
  A(\Lambda^{-2(n-1)} \theta) \cdots A(\theta) - T(\Lambda^{-2(n-1)} \theta) \cdots T(\theta).
\end{equation}
So $| r_{kl}^{(n)}(\theta) | \le \Psi(\theta)$, for all $n \in \Z^+$. Therefore, given a 
convenient explicit bound $\Psi_q(\theta) > \Psi(\theta)$, as presented in Remark 
\ref{rk-psiq}, we have:
\begin{equation} \label{est}
\begin{split} 
  & \left| \frac{\ell(\Lambda^{-2n} \theta_2)} {\ell(\Lambda^{-2n} \theta_1)} \cdot 
  \frac{\ell(\theta_1)} {\ell(\theta_2)} \right| \\
  & \quad > \frac{ 1 - (1 + 1/B_2^-) \Psi_q(\theta_2) } { 1 + (1 + 1/B_1^-) \Psi_q(\theta_1) } 
  \cdot \frac{ 1 - (1 +  B_1^+) \Psi_q(\theta_1) - B_1^+ (1-\Lambda^{-2\anot})^{-1}
  (\theta_1/\Lambda^2 )^\anot } { 1 + (1 +  B_2^+ ) \Psi_q(\theta_2) + B_2^+ 
  (1-\Lambda^{-2\anot})^{-1} (\theta_2/\Lambda^2 )^\anot }.
\end{split}
\end{equation}
Suppose one is now able to certify that the above r.h.s., which is explicit and
does not depend on $n$, is bigger than some $B \in (B_1^+/B_2^-, 1)$. They 
would conclude that 
\begin{equation} \label{end}
  \liminf_{n \to \infty} \left| \frac{\ell(\Lambda^{-2n} \theta_2)} {\ell(\Lambda^{-2n} \theta_1)} 
  \right| = \left| \frac{\ell(\theta_2)} {\ell(\theta_1)} \right| \cdot \liminf_{n \to \infty} 
  \left| \frac{\ell(\Lambda^{-2n} \theta_2)} {\ell(\Lambda^{-2n} \theta_1)} \cdot 
  \frac{\ell(\theta_1)} {\ell(\theta_2)} \right| \ge \frac{B_2^-} {B_1^+} B > 1,
\end{equation}
which implies, via definition \eqref{def:ctheta}, that $c_{\theta_1} \ne c_{\theta_2}$.

\medskip

Let us carry out this program (minus the certification of our numerics) 
for several values of $\Lambda$ and $\anot$, thus practically proving
Proposition \ref{prop-conj}. 

We first observe that there are values of $\theta$ for which $\ell(\theta)$ is explicitly 
known. In fact, $Z$ takes values in $(1+\Lambda^2) + (\Lambda^2 +
\Lambda^4)\N$. (This readily follows from the definition of $Z$, cf.\ Section 
\ref{SecStableZ}, and the assumptions $\pupup = 
\pdown + \pup = 1$: for 1-level vertical excursions of the 
walker, $Z = 1+\Lambda^2$; for 2-level excursions, $Z$ can take the values 
$1+\Lambda^2 +\Lambda^4+ \Lambda^2$ or $1+\Lambda^2+\Lambda^4
+\Lambda^2+\Lambda^4+\Lambda^2$, etc.) Therefore, for $k\in\mathbb Z$,
\begin{equation}
  \varphi_Z\! \left( \frac{2\pi k} {\Lambda^2 + \Lambda^4} \right) = 
  e^{i2\pi k/\Lambda^2} .
\end{equation}
Equivalently,
\begin{equation} \label{est-40} 
  \ell\! \left( \frac{2\pi k} {\Lambda^2 + \Lambda^4} \right) = \left( 1 - 
  e^{i2\pi k/\Lambda^2} \right)\left( \frac{2\pi k} {\Lambda^2 + \Lambda^4} 
  \right)^{-\anot} .  
\end{equation}
One could in principle choose two positive integers $k_1, k_2$ and set
$\theta_j := 2\pi k_j / (\Lambda^2 + \Lambda^4)$ ($j=1,2$), provided that:
\begin{itemize}
  \item $\theta_1/\theta_2$ is not a power of $\Lambda^2$, for in that case $c_{\theta_1} 
  = c_{\theta_2}$ by definition of $c_\theta$;
  \item $0 < |\ell(\theta_1)| < |\ell(\theta_2)|$, which is required by the above scheme,
  cf.\ the choice of $B$ and (\ref{end}). (Of course, the real conditions
  are $|\ell(\theta_1)| \ne |\ell(\theta_2)|$ and $|\ell(\theta_j)| \ne 0$, as $\theta_1$ and $\theta_2$ 
  can be interchanged.)
\end{itemize}
In this case, the bounds in (\ref{cert-est}) can be taken infinitely tight: 
$B_j^\pm := |\ell(\theta_j)|$. The problem, however, is that, even for small 
values of $k_j$, the value of $\theta_j$ is nowhere near as small as for the 
r.h.s.\ of (\ref{est}) to admit a usable lower bound $B$, for this requires 
$\Psi_q(\theta_j)$ to be quite small, which in turn requires $\theta_j$ to be quite 
small too.

On the other hand, starting from known values of $\theta$ and $\ell(\theta)$, one
can use (\ref{fundamental}) to compute, symbolically and/or numerically,
as many values of $\ell(\Lambda^{-2n} \theta)$ ($n \in \N$) as one's computing
power permits, to try and apply the above scheme to some of them, in the 
regime where $\Lambda^{-2n} \theta$ is small enough. 

Here we have followed this specific procedure. First of all, observing
that the term $\Lambda^{-2\anot}$ is ubiquitous in our computations,
cf.~Lemma \ref{lem-est} and Remark \ref{rk-psiq},
we have expressed our numerics in terms of the parameter
$\beta := \Lambda^{-2\anot} = \Lambda^{-\a} \in (\Lambda^{-2}, 1)$,
which is equivalent to $\anot \in (0,1)$. For $\Lambda = 2,3$ and 
several values of $\beta$, we have chosen suitable 
pairs of the form $s_j := 2\pi k_j / (\Lambda^2 + \Lambda^4)$ ($j=1,2$), 
for some $k_j \in \Z^+$. For increasing values of $n$, we have numerically 
checked whether the pair $\theta_j := \Lambda^{-2n} s_j$ would make our 
scheme work (namely, an accurate lower estimate of the r.h.s.\ of 
(\ref{est}) is bigger that an accurate upper estimate of $| \ell(\theta_1)/\ell(\theta_2) |$). 
We have always found a (minimum) value of $n$ that achieved the goal, 
and reported all the corresponding data in the tables below.

\medskip

\paragraph{\underline{Cases $\Lambda=2$}} For all these cases we have 
chosen $s_1 := \pi/5$, $s_2 := 2\pi/5$, corresponding respectively to 
$k = 2,4$ in (\ref{est-40}), whence $\ell(s_1) = 2 (\pi/5)^{-\anot}$,
$\ell(s_2) = 0$. The values of $\theta_j = 4^{-n} s_j$ reported 
below are given by the smallest $n$ for which the scheme works, as 
described above.

\medskip

\begin{center}
\small
\renewcommand{\arraystretch}{1.4}
\begin{tabular}{|c|c|c|c|c|c|c|c|}
  \hline
  $\beta$ & $\anot$ & $\theta_1$ & $\theta_2$ & $\ell(\theta_1)$ & $\ell(\theta_2)$ & 
  $\ds \left| \frac{\ell(\theta_1)}{\ell(\theta_2)} \right| $ & 
  \begin{minipage}{36pt} 
  \vspace*{2pt}
  \begin{center}
    r.h.s.\ of (\ref{est}) $^{(q=2)}$
  \end{center}
  \end{minipage} \\
  \hline\hline
  0.27 & 0.94448 & $4^{-191} \frac \pi 5$ & $4^{-191} \frac{2\pi} 5$ & 
  $4.7468 - 56.570i$ & $5.4773 - 56.865i$ & 0.99372 & 0.99407 \\
  \hline
  0.3 & 0.86848 & $4^{-65} \frac \pi 5$ & $4^{-65} \frac{2\pi} 5$ & 
  $3.9459 - 19.540i$ & $4.5780 - 19.765i$ & 0.98258 & 0.98526 \\
  \hline
  0.4 & 0.66096 & $4^{-18} \frac \pi 5$ & $4^{-18} \frac{2\pi} 5$ & 
  $2.30871 - 4.0115i$ & $2.6873 - 4.0863i$ & 0.94635 & 0.96133 \\
  \hline
  0.5 & 0.5 & $4^{-10} \frac \pi 5$ & $4^{-10} \frac{2\pi} 5$ & 
  $1.4283 - 1.4456i$ & $1.6410 - 1.4570i$ & 0.92605 & 0.93946 \\
  \hline
  0.6 & 0.36848 & $4^{-8} \frac \pi 5$ & $4^{-8} \frac{2\pi} 5$ & 
  $0.89810 - 0.60217i$ & $1.0150 - 0.60137i$ & 0.91651 & 0.92128 \\
  \hline
  0.7 & 0.25729 & $4^{-11} \frac \pi 5$ & $4^{-11} \frac{2\pi} 5$ & 
  $0.54303 - 0.23698i$ & $0.59588 - 0.22898i$ & 0.92814 & 0.94599 \\
  \hline
  0.8 & 0.16096 & $4^{-18} \frac \pi 5$ & $4^{-18} \frac{2\pi} 5$ & 
  $0.29541 - 0.076955i$ & $0.31356 - 0.071333i$ & 0.94930 & 0.95659 \\
  \hline
  0.9 & 0.076002 & $4^{-43} \frac \pi 5$ & $4^{-43} \frac{2\pi} 5$ & 
  $0.12122 - 0.014441i$ & $0.12461 - 0.012843i$ & 0.97450 & 0.97650 \\
  \hline
\end{tabular}
\end{center}

\medskip

\paragraph{\underline{Cases $\Lambda=3$}} For these cases,
$\{ s_1, s_2 \} = \{ \pi/5, 2\pi/5 \}$, corresponding to $k = 9, 18$ in 
(\ref{est-40}) and thus giving $\ell(s_1) = \ell(s_2) = 0$. The 
assignments of $s_j$ has been chosen in each case so that, for the 
resulting $\theta_j = 9^{-n} s_j$, $|\ell(\theta_1)| <  |\ell(s_2)|$. 
Once again, the selected value of $n$ is the smallest for which the 
scheme works.

\medskip

\begin{center}
\small
\renewcommand{\arraystretch}{1.4}
\begin{tabular}{|c|c|c|c|c|c|c|c|}
  \hline
  $\beta$ & $\anot$ & $\theta_1$ & $\theta_2$ & $\ell(\theta_1)$ & $\ell(\theta_2)$ & 
  $\ds \left| \frac{\ell(\theta_1)}{\ell(\theta_2)} \right| $ & 
  \begin{minipage}{36pt} 
  \vspace*{2pt}
  \begin{center}
    r.h.s.\ of (\ref{est}) $^{(q=2)}$
  \end{center}
  \end{minipage} \\
  \hline\hline
  0.13 & 0.92854 & $9^{-91} \frac{2\pi} 5$ & $9^{-91} \frac \pi 5$ & 
  $7.1152 - 54.501i$ & $4.5850 - 55.227i$ & 0.99180& 0.99185 \\
  \hline
  0.2 & 0.73249 & $9^{-24} \frac \pi 5$ & $9^{-24} \frac{2\pi} 5$ & 
  $2.8396 - 9.7361i$ & $4.4837 - 9.1033i$ & 0.99942& 0.99959 \\
  \hline
  0.3 & 0.54795 & $9^{-9} \frac \pi 5$ & $9^{-9} \frac{2\pi} 5$ & 
  $1.8538 - 3.6337i$ & $2.7803 - 3.0738i$ & 0.98422 & 0.98490 \\
  \hline
  0.4 & 0.41702 & $9^{-8} \frac \pi 5$ & $9^{-8} \frac{2\pi} 5$ & 
  $1.3517 - 1.8439i$ & $1.8460 - 1.3854i$ & 0.99053 & 0.99636 \\
  \hline
  0.5 & 0.31546 & $9^{-9} \frac{2\pi} 5$ & $9^{-9} \frac \pi 5$ & 
  $1.2385 - 0.66371i$ & $1.0090 - 0.99722i$ & 0.99050 & 0.99464 \\
  \hline
  0.6 & 0.23249 & $9^{-9} \frac{2\pi} 5$ & $9^{-9} \frac \pi 5$ & 
  $0.81913 - 0.31535i$ & $0.73372 - 0.52578i$ & 0.97239 & 0.97407 \\
  \hline
  0.7 & 0.16233 & $9^{-12} \frac{2\pi} 5$ & $9^{-12} \frac \pi 5$ & 
  $0.51139 - 0.13439i$ & $0.49445 - 0.24327i$ & 0.95954 & 0.96633 \\
  \hline 
  0.8 & 0.10156 & $9^{-19} \frac{2\pi} 5$ & $9^{-19} \frac \pi 5$ & 
  $0.28495 - 0.045883i$ & $0.28805 - 0.087170i$ & 0.95902 & 0.96697 \\
  \hline 
  0.9 & 0.047952 & $9^{-42} \frac{2\pi} 5$ & $9^{-42} \frac \pi 5$ & 
  $0.11951 - 0.0089062i$ & $0.12198 - 0.017123i$ & 0.97295 & 0.97430 \\
  \hline 
\end{tabular}
\end{center}

\medskip

All computations have been performed by MATLAB (rel.\ 2024b) with 
128-digit variable-precision arithmetic. All reals are presented with 5 
significant digits. Higher values of $q$ have been tried for the estimate 
of the r.h.s.\ of (\ref{est}) with little to no improvement.

\end{document}